\documentclass[11pt,a4paper]{article}

\usepackage[T1]{fontenc}
\usepackage{epsfig}
\usepackage{amsmath, amsthm}
\usepackage{amsfonts,amssymb}
\usepackage{color}
\usepackage{derivative}
\usepackage{mathtools}
\usepackage[applemac]{inputenc}		
\usepackage{hyperref}
\usepackage{comment}

\hypersetup{colorlinks=true,urlcolor=black, pdftitle=""}

\newtheorem{conj}{Conjecture}[section]
\newtheorem{thm}[conj]{Theorem}

\newtheorem{lem}[conj]{Lemma}
\newtheorem{prop}[conj]{Proposition}

\newtheorem{cor}[conj]{Corollary}
\newtheorem{fact}[conj]{Fact}

\makeatletter
\newtheorem*{rep@theorem}{\rep@title}
\newcommand{\newreptheorem}[2]{%
\newenvironment{rep#1}[1]{%
 \def\rep@title{#2 \ref{##1}}%
 \begin{rep@theorem}}%
 {\end{rep@theorem}}}
\makeatother
\newreptheorem{theorem}{Theorem}
\newreptheorem{corollary}{Corollary}
\newreptheorem{proposition}{Proposition}

\newcommand{\conv}{\mathrm{conv}}

\newcommand{\vol}{\mathrm{Vol}}

\newcommand{\interior}{\mathrm{int}}

\newcommand{\supp}{\mathrm{supp}}

\newcommand{\R}{\mathbb{R}}
\newcommand{\N}{\mathbb{N}}

\def\s{\mathbb{S}}

\def\phi{\varphi}

\def\bee{\begin{eqnarray*}}
\def\ene{\end{eqnarray*}}

\newcommand\nnfootnote[1]{%
  \begin{NoHyper}
  \renewcommand\thefootnote{}\footnote{#1}%
  \addtocounter{footnote}{-1}%
  \end{NoHyper}
}

\newcommand{\sant}{s_p}
\newcommand{\lap}{\mathcal L_p}

\newcommand{\intn}{\int_{\R^n}}

\newcommand{\domain}{{\rm dom}}

\newcommand{\intco}{{\rm int\, co\,  }}

\newcommand{\inte}{{\rm int}}

\begin{document}

\title{On a Santal\'o point for Nakamura-Tsuji's Laplace transform inequality}

\author{Dario Cordero-Erausquin, Matthieu Fradelizi, Dylan Langharst\thanks{Supported by the FSMP Post-doctoral program.}, 
}

\date{\today}
\maketitle

\begin{abstract}
Nakamura and Tsuji recently obtained an integral inequality involving a Laplace transform of even functions that implies, at the limit, the Blaschke-Santal\'o inequality in its functional form. Inspired by their method, based on the Fokker-Planck semi-group, we extend the inequality to non-even functions. We consider a well-chosen centering procedure by studying the infimum over translations in a double Laplace transform. This requires a new look on the existing methods and leads to several observations of independent interest on the geometry of the Laplace transform. Application to reverse hypercontractivity is also given. 

\end{abstract}
\nnfootnote{Keywords: Laplace transform, heat flow, Gaussian measure, volume product, Santal\'o point, Blaschke-Santal\'o inequality, hypercontractivity. 

Mathematics Subject Classification 2020 - Primary: 52A40 and 44A10.
Secondary: 35K05 and 52A21.
}


\section{Introduction and main results}

Our central object of interest is the Laplace transform $Lf$ of a nonnegative $f:\R^n\to \R^+$, defined by 
\begin{equation} \label{eq:laplace}
Lf(y):=\int_{\R^n} f(x)\, e^{x\cdot y}\, dx, \qquad  \forall y\in \R^n.
\end{equation}
It is implicit that we work with the standard $(\R^n, \cdot, |\cdot|)$ Euclidean structure and with (nonnegative) Borel measurable functions only.  We are looking at $L^p-L^q$ bounds for this transform, in the range $0<p<1$, where interactions with convex geometry take place. So, in the rest of the paper, our parameters will be as follows:
\begin{equation}\label{eq:condpq}
p\in(0,1) , \quad q= \frac{p}{p-1}\in (-\infty, 0).   
\end{equation}
Accordingly, $\|Lf\|_{L^q(\R^n)}>0$ amounts to some 'positivity' improvement, in the sense of Borell~\cite{Bor82}. Of course, we are mainly interested in functions that are not identically zero (almost everywhere), a case we denote by $f\equiv 0$. In other words, we write
$$f\not\equiv 0 \Leftrightarrow \{f>0\} \textrm{ has not measure zero}\Leftrightarrow \int f \in (0, \infty] \Leftrightarrow Lf \in (0, \infty].$$
Note that $Lf$ is a convex function with values in $[0, \infty]$, and  when 
$f\not\equiv 0$, one has that $\log Lf$ is a convex function with values in $\R\cup\{\infty\}$, by H\"older's inequality. In particular, $(Lf)^q$ is then a log-concave function, a property that will play a central r\^ole in our study. Centered Gaussian functions, that is functions of the form $\lambda\cdot\,  \gamma_\sigma$, $\lambda, \sigma>0$, where
\begin{equation} 
\label{eq:gauss}
\gamma_\sigma (x):=e^{-|x|^2/(2\sigma)}, \qquad \forall x\in \R^n,  
\end{equation}
will also be central. Let us denote by $\tau_a f$  the translate of a function $f$, defined on $\R^n$, by a vector $a\in \R^n$:
$$\tau_a f (x) := f(x-a), \qquad \forall x\in \R^n.$$ 
With this notation in hand, we can state one of our results. 

\begin{thm}[Reverse $L^p$ bound for the Laplace transform]
\label{t:laplace}
For any $f\in L^p(\R^n)$ nonnegative, we have
$$\sup_{z\in \R^n} \|L(\tau_z f)\|_{L^q(\R^n)} \ge C_p\, \|f\|_{L^p(\R^n)},$$
where $C_p=[p^\frac{1}{p}\left(-q\right)^{-\frac{1}{q}}]^{\frac{n}{2}} \, (2\pi)^{\frac{n}{q}}$
is so that there is equality when
$f$ is a (centered) Gaussian function. 
If the supremum is finite, it is uniquely attained at some point $z$, and this point is zero if
$\int_{\R^n} x\,  L(f)^q(x) \, dx=0$. 

Moreover, for any $f\in L^p(\R^n)$ nonnegative with $\intn x \, f(x)^p \, dx = 0$, we have 
\begin{equation}
\label{eq:laplace_lp_center}
\|Lf\|_{L^q(\R^n)} \ge C_p\,  \|f\|_{L^p(\R^n)},    
\end{equation}
with equality when $f$ is a centered Gaussian. 
\end{thm}
Nakamura and Tsuji~\cite{NT24} have proved~\eqref{eq:laplace_lp_center} in the case where $f$ is even. Generalizing it to non-even functions, under appropriate centering, will require substantial work. But before going forward, we need to step back and explain the motivation for our results which stems from convex geometry.

The Blaschke-Santal\'o inequality asserts that for every convex body (compact, convex set with non-empty interior) $K\subset \R^n$ that is symmetric (i.e. $K=-K$) one has,
\begin{equation}\label{eq:s1}
\vol(K)\vol(K^\circ)    \le \vol(B_2^n)^2,
\end{equation}
where $K^\circ=\{y\in \R^n \; ; \; x\cdot y\le 1, \forall x\in K\}$ and $B_2^n=\{|\cdot|\le 1\}$. See \cite{SLA49,MP90} for the inequality.

K. Ball~\cite{Ball_thesis} put forward in his thesis a functional form of the Blaschke-Santal\'o inequality: given a convex function $\psi:\R^n\to \R\cup\{+\infty\}$, with $0<\int e^{-\psi}<\infty$, it holds, when $\psi$ is even, that
\begin{equation}\label{eq:funct_s1}
M(e^{-\psi}):=\int_{\R^n} e^{-\psi} \int_{\R^n} e^{-\psi^\ast} \le (2\pi)^{n}=M(e^{-|x|^2/2}),
\end{equation}
where $\psi^\ast$ denotes the Legendre's transform of $\psi$, 
$$\psi^\ast(y):=\sup_{x} x\cdot y - \psi(x) , \qquad \forall y\in \R^n.$$
 Note that $\psi(x)=|x|^2/2$ is the unique fixed point for the Legendre's transform. Furthermore, we can also remove the assumption that $\psi$ is convex, since we have in general $\psi^{\ast \ast}\le \psi$ and $\psi^{\ast\ast \ast} = \psi^\ast$.

 One may actually prefer to work with log-concave functions. A function $f:\R^n\to [0,\infty)$ is log-concave if $f=e^{-\psi}$, for some convex function $\psi:\R^n\to \R\cup\{+\infty\}$. To any nonnegative $f$, we can associate its polar by
\begin{equation}
f^{\circ}(y):=\inf_{x\in\R^n}\frac{e^{- x\cdot y}}{f(x)} = e^{-\psi^\ast(y)},
\label{eq:polar}
\end{equation}
which is always a log-concave function; \eqref{eq:funct_s1} becomes for an even nonnegative function $f$ with $0<\int f <\infty$,
\begin{equation}\label{eq:funct_s2}
M(f)=\int_{\R^n} f \ \int_{\R^n} f^\circ \le (2\pi)^{n}=M(e^{-|x|^2/2}).
\end{equation}
One can easily pass from~\eqref{eq:funct_s2} to~\eqref{eq:s1}; we refer to~\cite{AAKM04,FM07} for background and references.

It was kind of folklore that one could try to approach the Legendre transform, or the polarity transform, by a limit of Laplace transforms, by interpreting  the '$\sup$' in the definition of $\psi^\ast$ as some $L^\infty$ norm (see for instance the discussion in T.~Tao's blog \cite{T07}; in the case of convex bodies, an $L^p$ approach was proposed by Lutwak and Zhang~\cite{LZ97}) and a Laplace transform approach was given by Berndtsson, Mastrantonis, and Rubinstein ~\cite{BMR24}. Let us also mention in passing that Klartag's works (e.g. in \cite{BK06}) have put forward the importance, in its own, of Laplace's transform  in the study of log-concave functions and convex bodies in high-dimensions. 

Recently, Nakamura and Tsuji~\cite{NT24} made the \emph{tour de force} of showing that such a path to the Blaschke-Santal\'o inequality is indeed possible by discovering a completely new, analytic, approach which gives in particular that $M(f)$ is non-decreasing when the even function $f$ evolves along the Fokker-Planck semi-group that interpolates between $f$ and a Gaussian function (see also \cite{CEGNT24}).
More precisely, they obtained, by evolution along the Fokker-Planck semi-group,  the following,  stronger, new, inequality for the Laplace transform, which gives at the limit (when $p\to 0^+$, $q\to 0^-$) the inequality \eqref{eq:funct_s2}: for a nonnegative  even function $f$, 
\begin{equation}\label{eq:NT}
\Big(\int_{\R^n}f \Big)\,  \Big(\intn \big(L(f^{1/p})(x/p)\big)^q\, dx\Big)^{-p/q} \le 
[p^{-p}(1-p)^{1-p}]^{\frac{n}{2}}\,
(2\pi)^{n(1-p)},
\end{equation}
with equality when $f=e^{-|x|^2/2}$. 
The term in~\eqref{eq:NT} with the Laplace transform together with the unusual range of the parameters~\eqref{eq:condpq} might be a bit hard to digest at first, but before we return to it,  we ask the reader to accept that it is not that bad and in fact quite natural. Of course, the normalization $x/p$ plays no serious role but is a convenient scaling when we want to let $p$ go to zero, as we will see.  

In this work, we aim at extending \eqref{eq:NT} to non-even functions; we establish in particular:

\begin{thm} \label{t:main_sant}
Let $f:\R^n\to [0,\infty)$ be a nonnegative function. Then, 
    $$ \Big(\intn f \Big)\, \inf_z \Big( \intn \mathcal (L(\tau_z f^{1/p})(x/p))^q \, dx \Big)^{-p/q} \le 
    [p^{-p}(1-p)^{1-p}]^{\frac{n}{2}}\, (2\pi)^{n(1-p)},$$
with equality for any Gaussian function.
\end{thm}

Here and in the rest of the paper we are consistent with monotone convergence and assume that $\infty\times 0=0$. 

In order to  shed some light on the previous result, let us recall how the non-symmetric Blaschke-Santal\'o inequality is handled. Indeed, the Blaschke-Santal\'o inequality cannot hold as stated in~\eqref{eq:s1} in the non-symmetric case. For a convex body $K$, it is easily checked that
$$\vol(K^\circ) = \infty \Longleftrightarrow 0 \notin {\rm int} K.$$
But just having zero in the interior is certainly not sufficient for the Blaschke-Santal\'o inequality to hold. Indeed, if $a$ is a point in $K$ that approaches the boundary of $K$, then $\vol((K-a)^\circ) \to \infty $. So in order to hope to have a bound for $\vol(K)\vol(K^\circ)$, one certainly needs to wisely translate the body $K$ before taking its polar. This indeed works: we have \begin{equation}
\label{eq:santalo_convex}
\inf_z \vol(K)\vol((K-z)^\circ)\le \vol(B_2^n)^2,\end{equation}
and, also, that the infimum is obtained at a unique point, the \textit{Santal\'o point of $K$}, characterized as the unique vector $s(K)$ so that $(K-s(K))^\circ$ has center of mass at the origin. So~\eqref{eq:s1} holds when $K^\circ$ has center of mass at the origin; but it turns out that it also holds when $K$ has center of mass at the origin. We refer to~\cite{CMP85,FGSZ24, FMZ23} for background and further developments. 

Analogously, in the functional case, since
\begin{equation}\label{eq:translate}
    (\tau_z f)^\circ(x)=f^\circ(x)\,  e^{-x\cdot z},
\end{equation}
we can make the left-hand side of~\eqref{eq:funct_s2} tend to infinity by letting $|z|\to \infty$. Therefore, a non-even function $f$ must be wisely translated for \eqref{eq:funct_s2} to hold.
Using the geometric version of the Blaschke-Santal\'o inequality for non-symmetric bodies, Artstein, Klartag and Milman \cite{AAKM04} proved, when $f$ is a nonnegative function with $0<\int_{\R^n} f < \infty$, that
\begin{equation}\label{eq:func_ns1}
M(f):=\left(\int_{\R^n} f\right)\inf_z\left(\int_{\R^n} (\tau_{z}f)^\circ \right) \leq (2\pi)^n.
\end{equation}
Actually, since $f\le f^{\circ \circ}$ (with equality when $f$ is log-concave and upper-semi continuous) one can reduce the study to log-concave functions. Then, when $f$ is log-concave, they established that the infimum in~\eqref{eq:func_ns1} is attained at a uniquely defined point $s(f) = \text{argmin}_{z\in\R^n} \int_{\R^n} (\tau_zf)^\circ $ which is called the Santal\'o point of $f$. Using~\eqref{eq:translate}
one can note that $s(f)$ is the origin, if, and only if, the barycenter of $f^\circ$ is the origin; this is certainly the case when $f$, and thus $f^\circ$, is even. Recall that the barycenter of a nonnegative function is given by $${\rm bar}(f):=\frac{\int_{\R^n}x\, f(x)dx}{\int_{\R^n} f(x)\, dx};$$
the condition $0<\intn (1+|x|) f(x)\, dx < \infty$ will always be implicitly enforced when we speak of barycenter of an arbitrary nonnegative function $f$. 

Analogously to the volume case, it was shown by  Lehec \cite{JL08,JL09} (see also \cite{FGSZ24}) that the Santal\'o point of $f$ can be replaced by the barycenter, that is, for a nonnegative function $f$ having its barycenter at zero, it holds that, 
\begin{equation}
\label{eq:barycenter_origin}
\Big(\int_{\R^n}f \Big)\,  \Big(\int_{\R^n} f^\circ \Big) \le (2\pi)^n.\end{equation}
This extends on the case shown by Artstein, Klartag and Milman \cite{AAKM04} for log-concave $f$ with $\int f\in (0,\infty)$. Generalizations of the functional Blaschke-Santal\'o inequality to multiple functions were recently purposed in \cite{KW22,KS23}.

Following Nakamura and Tsuji, we now consider the $p$-Laplace transform of a non-negative function $f$, where $p,q$ are in the range~\eqref{eq:condpq}, defined by 
$$\lap (f)(x) := (L(f^{1/p}))^q (x) = \Big( \intn f^{1/p}(y)\,  e^{x\cdot y}\, dy\Big)^q,
\qquad \forall x\in \R^n
.$$
To avoid trivial situations and abstract nonsense, we will most of the time assume that $f\not\equiv 0$, since:
$$\textrm{either $\lap(f)$  is finite on $\R^n$ (when $f\not\equiv 0$), \  or else $\lap(f)\equiv \infty$ (when $f\equiv 0$)}.$$
Since $q<0$, the function $\lap (f) = e^{q\log Lf^{1/p}}$ is log-concave when $f\not\equiv 0$. Actually, we will work with $\lap (f)(x/p)$, but
the scaling $x/p$ is a bit immaterial when we work at fixed $p$. Its justification comes from the limit when $p\to 0^+$ since we have for almost every $x\in \R^n$ (see~Fact~\ref{f:supp_converge} for the precise statement) that 
\begin{equation}
\label{eq:limit_p}
\begin{split}
\lim_{p\to 0^{+}}\lap(f)(x/p)&=\lim_{p\to 0^+}\left(\int_{\R^n} \left(e^{x\cdot y }f(y)\right)^\frac{1}{p}dy\right)^\frac{p}{p-1}
\\
&={\rm ess}\!\inf_{y\in\R^n}\frac{e^{-x\cdot y}}{f(y)}\ge \inf_{y\in\R^n}\frac{e^{-x\cdot y}}{f(y)} = f^{\circ}(x).
\end{split}
\end{equation}
Let us mention here another analogy with the polar transform:
\begin{equation}\label{eq:trans}
    \lap(\tau_z f)(x)=\lap (f) (x) \, e^{q x\cdot z},
\end{equation}
which should be compared to~\eqref{eq:translate}. From \eqref{eq:trans}, one obtains
$$\intn \lap(\tau_z f)(x)dx = L(\lap(f))(qz),$$
which is somehow a kind of double Laplace transform. Note also that
\begin{equation}\label{eq:trans2}
    \lap(f(y)\,  e^{z\cdot y})=\tau_{-z/p}\lap (f) .
\end{equation}
We define the \textit{$L^p$ functional volume product} of a nonnegative function $f$ as
\begin{eqnarray*}
M_p(f)&:=&\Big(\int_{\R^n}f \Big)\, \inf_z \Big(\int_{\R^n}{\mathcal  L}_p (\tau_z f)(x/p)\, dx \Big)^{-p/q} \\
&= & p^{-np/q}  \Big(\int_{\R^n}f \Big)\, \inf_z  \Big(\int_{\R^n}{\mathcal  L}_p (\tau_z f) \Big)^{-p/q} \\
& =&  p^{-np/q}  \Big(\int_{\R^n}f \Big)\, \Big(\inf L \lap(f)\Big)^{-p/q}.
\end{eqnarray*}
Note that $M_p$ is homogeneous and linear invariant, in the sense that, for any constant $\lambda>0$ and any invertible linear application $A\in GL_n(\R)$, we have
\begin{equation}\label{eq:inv}
M_p(f)=M_p(f_{\lambda,A}), \quad {\rm where}
\ 
f_{\lambda, A}(x) = \lambda f(Ax).
\end{equation}
Its action over the Gaussian functions~\eqref{eq:gauss} is given by
\begin{equation*} 
\label{eq:limit_fixed}
{\mathcal  L}_p (\gamma_\sigma)=c_p ^\frac{n}{2}\, \gamma_{\frac1{-\sigma p q}},\quad c_p = (\sigma p2\pi)^q. 
\end{equation*}
This suggest other possible scalings; for instance $\gamma_1$ is an eigenfunction of  $f\mapsto \lap(f)(x/\sqrt{-pq})$, but the limit process is simpler with $x/p$. 
Anyway, $M(\gamma_\sigma)=M(\gamma_1)$ for any $\sigma$. The invariance property~\eqref{eq:inv} suggests to ask for the supremum of $M_p$; this is precisely the content of our Theorem~\ref{t:main_sant} above:
for any nonnegative function $f$, 
    $$ M_p(f) \le M_p(\gamma_1).$$
Note that, in particular, if $\int_{\R^n}f = \infty$, then $\inf_z \int_{\R^n}\lap(\tau_z f) =0$.

Pushing forward the analogy with the Blaschke-Santal\'o inequality, we would like to know whether the infimum over translations is attained. It is the case in non-pathological situations. Actually, the infimum is always finite,  and we have a nice dichotomy.

\begin{prop}[$p$-Laplace-Santal\'o point]\label{prop:sant}
Let $f:\R^n\to [0,\infty)$ be an nonnegative function with $f\not\equiv 0$ and  consider the convex function
$$\psi(z):= \int_{\R^n} {\mathcal L}_p(\tau_z f)
=\int_{\R^n}{\mathcal L}_p(f)(x)\, e^{q \, x\cdot z}\, dx = L\lap(f)(qz).$$ 
Then $\psi$ is always proper, that is there exists $z_0$ such that $\psi(z_0)<\infty$ or equivalently $\inf \psi <\infty$. 

Moreover, either
$$\inf_{z\in \R^n} \psi(z) = 0$$
or else $\psi$ attains its (non-zero) minimum, at a  unique point that we denote by
\begin{equation} \label{eq:deflpoint}
\sant(f) := {\rm argmin}\Big\{ \int {\mathcal L}_p(\tau_zf) \; ; \; z\in \R^n\Big\}.
\end{equation}
The point $\sant(f)$ is then characterized by the property that
$${\rm bar} \big({\mathcal L}_p(\tau_{\sant(f)}f)\big)=0.$$
\end{prop}

If $f$ is even and $\lap(f)$ is integrable, it follows that $\sant(f)$ is equal to zero, by the previous proposition, since $\lap(f)$ is even (see Proposition~\ref{p:even}). Thus,  Theorem~\ref{t:main_sant} implies \eqref{eq:NT}.

Interestingly enough, like in the case of the functional the Blaschke-Santal\'o inequality, the inequality is valid when $f$ has barycenter at the origin.

\begin{thm}
\label{t:main_sant_2}
Let $f:\R^n\to [0,\infty)$ be a nonnegative function. If either $f$ or $\lap(f)$ has $0$ as barycenter, then we have
$$M_p(f)=\Big(\int_{\R^n}f \Big)\,  \Big(\int_{\R^n}{\mathcal  L}_p (f)(x/p)\, dx \Big)^{-p/q} \le M_p(\gamma_1),$$
with equality when $f$ is a centered Gaussian. 
\end{thm}

Note that, by letting $p\to 0^+$, and combining~\eqref{eq:limit_p} with Fatou's Lemma, we recover \eqref{eq:barycenter_origin}. We can even replace the polar by the essential polar defined by~\eqref{eq:ess_polar}.

Observe also that Theorem~\ref{t:laplace} follows directly from Theorems~\ref{t:main_sant} and~\ref{t:main_sant_2} applied to the nonnegative function $f^p$. Indeed, we have
$$\|L(\tau_z f)\|_q^q =  L\lap(f^p) (qz),$$
and so $\sup_z \|L(\tau_z f)\|_q= (\inf L\lap(f^p))^{1/q}$  since $q<0$. Note for consistency that,  from Proposition~\ref{prop:sant},   $\sup_z \|L(\tau_z f)\|_q \in (0, \infty] $ as soon as $f\not \equiv 0$.

As in Nakamura-Tsuji's work, our main theorem will be proved by letting $f$ evolve along the Fokker-Planck semi-group $(P_tf)$ given by
$$\partial_t P_t f = \frac{1}{2}(\Delta P_tf +{\rm div}(x P_t f)),$$
which can be used to interpolate between $P_0f=f$ and the Gaussian function $$P_\infty f(x)=(2\pi)^{-n/2}\left(\int_{\R^n} f\right) e^{-|x|^2/2}.$$ 
We will show that, for fixed $p$ and $f$,
$M_p(P_tf)$ then increases in time. Note that we really care about the evolution of the infimum, since
$$M_p(P_t f)=  p^{-np/q}  \Big(\int_{\R^n}f \Big)\, \Big(\inf L \lap(P_t f)\Big)^{-p/q}.$$
We will rather work with the heat flow, $E_0 f=f $ and $\partial E_t f = \frac{1}{2}\Delta E_tf$, which is totally equivalent in our situation, but leads to simpler computations. The semi-group evolution will be defined by integral formulas that make sense for any nonnegative function $f$. The next result is the main result of the paper, from which Theorem~\ref{t:main_sant} immediately follows.

\begin{thm}
\label{t:main_2}
    Let $f$ be a nonnegative function and let $f_t$ be its evolution along the Fokker-Planck or heat semi-group. 
    Then, the function $t\mapsto M_p(f_t)$ is increasing in $t\in [0, \infty)$ and is dominated by $M_p(\gamma_1)$.
\end{thm}

By letting $p\to 0^+$ we obtain the following result regarding the volume product. Here we use the notation~\eqref{eq:ess_polar} for the 'essential' polar, denoted by $f^\square$, which dominates the usual one. 
\begin{cor}
\label{cor:volume_heat}
    Let $f$ be a nonnegative function on $\R^n$. Let $f_t$ denote its evolution along the heat or Fokker-Planck flows. Then, for $t\geq 0$, the function $$t\mapsto M(f_t)=\left(\intn f \right)\, \inf_z\left(\intn (\tau_z f_t)^\square\right)$$ is increasing on $[0,\infty)$ and dominated by $M(\gamma_1)$.
\end{cor}

One can combine the ideas of the present paper with those of~\cite{CEGNT24} to provide a more direct proof of Corollary~\ref{cor:volume_heat} when $f$ is log-concave. 

We mention here that there are also $L^p$ extensions of the Blaschke-Santal\'o inequality for convex bodies. In particular,  Lutwak and Zhang~\cite{LZ97}  have studied $p$-centroid bodies and established sharp inequalities  that imply, at the limit $p\to \infty$, the Blaschke-Santal\'o inequality~\eqref{eq:s1}; see also~\cite{HS09,HLPRY23_2}. Recently, Berndtsson, Mastrantonis and Rubinstein in the symmetric case \cite{BMR24} and Mastrantonis \cite{VM24} in general considered a $L^p$ volume product $\mathcal{M}_p(K)$ among convex bodies for $p\in [0,\infty]$, which in our notation is $\mathcal{M}_{\frac{1-p}{p}}(K) = \left(\frac{p}{1-p}\right)^{n}M_p(1_K)^{1-p}$ and satisfies $\mathcal M_\infty(K)=M(K)$. They showed $\mathcal{M}_p(K) \leq \mathcal{M}_p(B_2^n),$ with equality only for centered ellipsoids. Note that applying our result to an indicator function would lead to a sub-optimal result, since our extremizers are Gaussian functions. We do not see how  to pass from the functional Nakamura-Tsuji type inequalities to the geometric inequalities from~\cite{BMR24,VM24}, or conversely, unlike what happens for the volume product. 
It would be of interest to investigate further connections between our functional inequalities and $\mathcal{M}_p(K)$, or variants of this quantity. 

After the present work was submitted 
and put on arXiv, Mastrantonis \cite{VM24_2}, working independently from us, defined a functional $L^p$ volume product via the relation $\mathcal{M}_{\frac{1-p}{p}}(f) := \left(\frac{p}{1-p}\right)^{n}M_p(f)^{1-p}$, and proved Proposition~\ref{prop:sant} and Theorem~\ref{t:main_2} for the Fokker-Planck heat semi-group when $f$ is log-concave. Working with log-concave functions led him to observations of independent interest and simplifications, for instance the restrictions in our Theorem~\ref{t:main} below are no longer needed. 

We conclude with an application to reverse hypercontractivity. Let $\gamma$ be the standard Gaussian measure on $\R^n$, that is $d\gamma(x)=\frac{\gamma_1(x)}{(2\pi)^{n/2}}dx$. The Ornstein-Uhlenbeck semi-group $(U_t g)_{t\ge 0}$ is given by, for a nonnegative function $g$ by
\begin{equation}
\label{eq:ornstein_uhlenbeck}
U_tg(x)=\int_{\R^n}g(e^{-t}x+\sqrt{1-e^{-2t}}\, z)d\gamma(z).
\end{equation}
The next result extends the even case obtained by Nakamura and Tsuji~\cite{NT24}, who emphasized that Laplace transform inequalities in the range~\eqref{eq:condpq} can be rephrased as reversed hypercontrative estimates for the Ornstein-Uhlenbeck operator. 
\begin{thm}
\label{t:main_3}
    Let $p,q$ be as in~\eqref{eq:condpq}. Define $s>0$ via the relation $p=1-e^{-2s}$, so that $q=1-e^{2s}$. Suppose $f$ is a nonnegative function such that either $$\int x\, f^{p}(x)\, d\gamma=0 \text{ or }\int x\,  U_s(f)^{q}(x) \, d\gamma =0.$$ Then, for every $p_2\geq q$ and $p_1\leq p$, it follows that
    \begin{equation}\label{eq:hyper}
    \|U_s f\|_{L^{p_2} (\gamma)} \ge   \|f\|_{L^{p_1}(\gamma)}.    
    \end{equation}
\end{thm}
Observe that $\frac{q-1}{p-1} = e^{4s}$ which improves over the usual Borell's reverse hypercontractivity~\cite{Bor82}, which holds for $\frac{q-1}{p-1} = e^{2s}$. Having improvement of constants or parameters in inequalities by a factor of two under symmetry is a common feature. Note, however, that in our result $p$ and $q$ are conjugate, which restricts the range of parameters with respect to usual hypercontrativity. We also mention it was observed in~\cite{NT22} that, under the assumption $\int x \, f^p(x) \, d\gamma(x)=0$, the Blaschke-Santal\'o inequality implies~\eqref{eq:hyper} for the parameters $p_1 =p=1-e^{2s}$ and $p_2=-p_1>q$, which satisfy the more restrictive condition $ e^{2s}<\frac{p_2-1}{p_1-1}<e^{4s}$ .

\medskip

The rest of the paper is organized as follows. 
\begin{itemize}
\item In Section~\ref{sec:laplace}, we carefully study the Laplace transform of log-concave functions, and more precisely the map $g\mapsto \inf Lg$
 for a log-concave function $g$,  and apply it to the log-concave function $g=\lap(f)$, in order to determine the existence (or non-existence) of the point obtaining the infimum. Then, we prove Proposition~\ref{prop:sant}. We also investigate the continuity properties of the infimum above that will prove crucial in all limiting and approximating processes later.  This section relies on convexity and the Hahn-Banach theorem will be used several times.
 \item In Section~\ref{sec:semi}, we recall the definitions of semi-groups and establish Theorem~\ref{t:main_2} in the particular case where $f$ is bounded and compactly supported. This section is about semi-groups and functional inequalities and is the core of our analytical arguments.  
 \item In Section~\ref{sec:proofs}, we give the proof of Theorem~\ref{t:main_2}, in whole generality,  and of Corollary~\ref{cor:volume_heat}. This relies on successive approximation arguments. It also contains the short derivation of Theorems~\ref{t:laplace} and~\ref{t:main_sant_2}. 
 \item The final section~\ref{sec:applications} contains the straightforward derivation of the result above for the Ornstein-Uhlenbeck semi-group and some questions on the $p$-Santal\'o 'curves'.
\end{itemize}


\section{Laplace transform, existence  and continuity 
of the Laplace-Santal\'o point}
\label{sec:laplace}

In this section, we investigate $L\lap(f)$, which involves a double Laplace transform, and prove Propositions~\ref{prop:sant} together with some useful results around it. Since $\lap(f)$ is log-concave (when $f\not\equiv 0)$, most of the arguments rely on properties of the Laplace transform of log-concave functions.

For a convex function $\psi:\R\to \R \cup\{\infty\}$, we denote by $\domain(\psi)=\{\psi < \infty\}$ its domain, which is non-empty when $\psi$ is proper, by definition.

Recall that, for a nonnegative $f:\R^n \to [0, \infty]$, we define its Laplace transform, at every $x\in \R^n$, by
$$Lf(x) = \int_{\R^n} f(y) \, e^{x\cdot y}\, dy . $$
When  $f\not\equiv 0$ a.e., we have  following properties:
\begin{itemize}
    \item $Lf\in (0, \infty]$ everywhere and $Lf$ is strictly convex by properties of the function $x\mapsto e^{x\cdot y}$.
    \item $\log(Lf)$ is, by H\"older's inequality, a convex function with domain
$$\domain(\log Lf)= \domain(Lf) = \{Lf < \infty\},$$
and $\log Lf$ is lower-semi-continuous (by Fatou's Lemma). 
\end{itemize}

Recall finally that, for a non-negative function $f:\R^n\to \R^+$, its (essential) support is the closed set defined by 
$$\supp(f)=\R^n \setminus \Big\{x\in \R^n \; ; \; \exists r>0, \ 
\int_{B(x,r)} f=0 \, \Big\},$$
where $B(x,r)$ is the ball of radius $r$ centered at $x$.

\subsection{More on Laplace transform of log-concave functions}
Recall that a function $g:\R^n \to [0, \infty)$ is said to be  log-concave if $g=e^{-\psi}$, where $\psi:\R^n \to \R\cup\{\infty\}$ is convex; 
then the interior of the support of $g$, which is also the interior of the domain of $\psi$, is a convex set that verifies
$$\textrm{int}(\textrm{supp}(g) )= \textrm{int} \{g>0\},$$
which is non-empty when $g\not\equiv 0$, and $g$ is continuous on this set. 

It was observed by Klartag~\cite[Lemma 2.1]{BK07}, under an assumption of integrability, that the domain of $\log Lg$ is an open set on which $\log Lg$ is strictly convex. In Fact~\ref{fact:mf} below,  we complete, in more generality,  the description of the behavior of $\log Lg$, and characterize when it attains its infimum in its domain. We begin with the following fact concerning integrability of log-concave functions, most of which is classical. We recall that, for a convex set $K$ in $\R^n$, its gauge is given by $\|y\|_K=\inf\{r>0:y\in rK\} \in [0, \infty]$.
\begin{fact}
\label{f:integrable}
    Let $g:\R^n\to [0, \infty)$
    be a log-concave function with $g\not\equiv 0$. The following properties are equivalent:
    \begin{enumerate}
        \item[(i)] $\intn g < \infty$.
        \item[(ii)] There exists some constants $A,B>0$ such  that for all $x\in \R^n$, 
\begin{equation}\label{eq:majo_log_concave}
g(x)\le A\, e^{-B\, |x|}.
\end{equation}
    \item[(iii)] For every $a\in\R^n$ and $u\in\s^{n-1}$, one has $g(a+tu) \to 0$ as $t\to+\infty$.
    \end{enumerate}
\end{fact}
Note that the degenerate situation, in dimension $\ge 2$, where $g$ is constant, say, along a line and zero elsewhere, is excluded by the condition $g\not\equiv 0$.

\begin{proof}
The equivalence $(i)\Leftrightarrow(ii)$ is easy and standard, see for instance~\cite{BK07}, and
$(ii)\Rightarrow(iii)$ is obvious. So we concentrate on $(iii)\Rightarrow(i)$. 

Note that in dimension one, if $\alpha:\R\to \R^+$ is a log-concave function such that $\alpha(t)$ tends to $0$ as $|t|\to \infty$ (equivalently, the convex function $t\mapsto -\log(\alpha(t))$ tends $\infty$ as $|t|\to\infty$), then there exists $C,c>0$ such that $\alpha(t) \le C \, e^{-c\, |t|}$ for all $t\in \R$. 
Fix $a\in \inte(\supp(g))$. Then,
    \begin{align*}
        \intn g(x) dx = \intn g(a+x) dx = \int_{\s^{n-1}}\left(\int_0^\infty t^{n-1}g(a+tu)dt\right)du.
    \end{align*}
The remark above with $\alpha(t)=g(a+tu)$  gives that $t\mapsto t^{n-1}g(a+tu)$ is  integrable on $\R^+$.
Since $g$ is $\log$-concave, there exists a convex set $K_a$ (see \cite[Theorem 3]{Ball88}, \cite[Corollary 4.2]{GZ98} or \cite[Theorem 3.1]{CEFPP15}) such that, for any $x\in\R^n$,
    $$\frac{1}{n}\|x\|_{K_a}^{-n} = \int_0^\infty t^{n-1}g(a+tx)dt.$$
From our choice of $a$, we have that $\|\cdot\|_{K_a}$ is finite, and therefore continuous since it is convex. Moreover, since $t^{n-1}g(a+tu)$ is integrable, $\|u\|_{K_a} >0$  for all $u\in\s^{n-1}$, and so there exists $c>0$ such that $\|u\|_{K_a} \ge c $  for all $u\in\s^{n-1}$. This implies that   $K_a$ is bounded and therefore 
     $$\int_{\R^n}g(x)dx = \frac{1}{n}\int_{\s^{n-1}}\|u\|_{K_a}^{-n} du = \vol_n(K_a) < \infty.$$
\end{proof}
We will assume first that the domain of $Lg$ is non-empty. We give a sufficient condition of this property later, which will always be verified in our applications. It is also worth noting that for a log-concave function $g$, the barycenter
$${\rm bar}(g):=\frac{\intn g(x) x\, dx}{\intn g(x)dx}$$
is always well defined as soon as $0<\intn g < \infty$. 

The next fact is  central for our arguments. 

\begin{fact}
\label{fact:mf}
    Let $g$ be a log-concave function with $g\not \equiv 0$. Assume $\domain(Lg)\neq\emptyset$. 
    Then:

    \begin{enumerate}
    \item The set $\domain( Lg)$ is an open and convex set, on which  $Lg$ is smooth and strictly convex, and we have
    \[\domain( Lg)\subseteq \{g^\circ>0\}\subseteq \overline{\domain( Lg)}. \] 

    \item 
If the origin is in the interior of the support of $g$, then $ Lg$ tends to $\infty$ at the boundary of $\domain( Lg)$. Consequently, $ Lg$ attains its minimum, at a unique point in $\domain(Lg)$. This point $s_0$ is characterized by the property
    $${\rm bar} (g(x) e^{x\cdot s_0})=0.$$
    
    \item On the other-hand, if the origin is not in the interior of the support of $g$, then there exists a vector $z_0\in\R^n$ and a direction $u\in\s^{n-1}$ such that $\lim_{t\to\infty}Lg (z_0+tu)=0$.

    \end{enumerate}
\end{fact}

\begin{proof}
When $g\not\equiv 0$ is integrable, it was easily noticed by Klartag \cite{BK07}, using \eqref{eq:majo_log_concave},  that the domain of the convex function $\log Lg$, which equals the domain of $Lg$, is \emph{open} and that the function $\log Lg$ is strictly convex and  smooth on its domain ; for this, one uses that the Hessian of $\log Lg$ is a covariance matrix that is strictly positive when the support of $g$ has non-empty interior, which, for a log-concave function $g$, amounts to $\int g>0$. It then follows that $Lg$ is strictly convex and smooth on its domain.

Denote by $\sigma_z(x)=e^{z\cdot x}$. Let $z_0\in \domain(Lg)$; then $g \sigma_{z_0}\in L^1(\R^n)$ and $L(g\sigma_{z_0})=\tau_{-z_0}Lg$.
Applying the previous observation to $g \sigma_{z_0}$, it follows that $Lg=\tau_{z_0} L(g\sigma_{z_0})$ is strictly convex and smooth on its domain which is open. 

Next, note that
$\{g^\circ >0\}= \{z\; ;\; g \sigma_z \in L^\infty(\R^n)\}$
while $\domain(Lg)=\{z \; ;\; \sigma_z g\in L^1(\R^n)\}$. As mentioned above, for a log-concave function, which is $\not \equiv 0$, being integrable implies being bounded, so we have $\domain(Lg)\subseteq \{g^\circ>0\}$.  As before, since $\sigma_{z_0} g$ is integrable,  there exists $a,b>0$ such that $g(x) e^{z_0\cdot x}
\le b e^{-a|x|}$, for all $x\in \R^n$. Take now $z\in \{g^\circ>0\}$. For $t\in (0,1)$, if we set $z(t) = (1-t) z_0 + t z$, then we have
$$g(x)\,  e^{z(t)\cdot x} \le \frac{b^{1-t}}{(g^\circ (z))^t} e^{-a (1-t)\,  |x|}, \qquad \forall x\in \R^n.$$
Thus $z(t) \in \domain(L g)$. Since $z(t) \to z$ as $t\to 1$, we conclude that $z\in \overline{\domain(Lg)}$.
We have established the first claim. We now study the behavior at the boundary.

Suppose the origin is in the interior of the support of $g$. Then, there exists $\varepsilon>0$ such that $2\varepsilon B_2^n$ is included in the interior of the support of $g$. Since $g$ is continuous on the interior of its support, it follows that there exists $c>0$ such that $\int_{x\cdot u>\varepsilon}g(x)dx\ge c$ for all $u\in \s^{n-1}$. Let $z\in\R^n$, with $z\neq0$. Then 
\[
Lg(z)\ge\int_{x\cdot z>\varepsilon|z|} g(x)\,dx\;  e^{\varepsilon|z|}\ge c\, e^{\varepsilon|z|}.
\]
It follows that $\lim_{|z|\to+\infty}Lg(z)=\infty$.

We now consider the possibility that the domain is not the whole space and we have a vector $a$ in the boundary, necessarily outside the domain since it is open. If $(a_n)$ is any sequence in the domain tending to $a$, then by Fatou's Lemma: 
$$\liminf \intn g(y)\, e^{a_n \cdot y} \, dy \ge \intn g(y)\,  e^{a\cdot y}\, dy = \infty. $$
As a consequence of this boundary behavior, we see that $\log Lg$ attains its minimum,  at a point in its domain that we denote by $s_0$. This point is unique since $\log Lg $ is strictly convex.
Moreover $\log Lg$ is differentiable in the interior of its domain, so the point $s_0$ is characterized by
$$\nabla \log Lg(s_0)=0,$$
which rewrites as
$$0=\intn x \, \frac{g(x)\, e^{x\cdot s_0}\, 
dx}{\intn g(x)e^{x\cdot s_0}\, dx}.$$ 
Recall here that an integrable log-concave function has moments of all orders, by~\eqref{eq:majo_log_concave}.

Finally,  if the support of $g$ does not contain the origin, then, from the Hahn-Banach theorem, there exists a direction $u\in\s^{n-1}$ and a hyperplane  $H=\{x\in\R^n\; ; \;  x\cdot u =0\}$ with outer-unit normal $u$ such that $\supp(g) \subseteq H^{-}=\{x\in\R^n\; ; \; x\cdot u \leq 0\}$. Using again $z_0\in \domain(Lg)$, we then have, for every $t>0$,
\[
    Lg(z_0+tu) = \int_{\R^n}e^{x\cdot (z_0+tu)}g(x)dx = \int_{\{x\cdot u < 0\}}e^{tx\cdot u}\, e^{x\cdot z_0}\,  g(x)dx.
\]
Since on $H^-$ we have $ e^{tx\cdot u}\, e^{x\cdot z_0}\,  g(x)\le e^{x\cdot z_0}\,  g(x)$ and that this upper bound is integrable, we can use dominated convergence and obtain
$$\lim_{t\to \infty}Lg(z_0+tu) = 0.$$
\end{proof}

The special role the origin plays in Fact~\ref{fact:mf} is analogous to the role the origin plays for the classical duality of a log-concave function. 

We now move to show a sufficient condition for the domain of the Laplace transform of a $\log$-concave function to be non-empty. We can do later without this, but we feel it is of independent interest. We say a function $F$ is affine along a line if  there exists $u,v\in\R^n$ and $\alpha,\beta\in\R$ such that, for every $t\in \R,$
        $$F(tu+v) = \alpha t + \beta.$$

\begin{fact}
\label{fact:finite}
Let $F:\R^n\to \R\cup\{\infty\}$ be a convex lower semi-continuous function. If $\domain(L(e^{-F}))=\emptyset$, then $F$ is affine along a line.
\end{fact}
\begin{proof} We can assume that $F$ is proper, for, if not, we would have $e^{-F}=0$ and so $\domain(L(e^{-F}))=\R^n$. 
Consider the epigraph of $F$: 
$$C=\{(x,t); t \geq F(x)\}\subset \R^n\times \R.$$
Then, $C$ is a non-empty closed, convex set. If $F$ is not affine along a line, then $C$ does not contain a line; indeed if the closed convex set $C$ contains a line $\ell$, then all lines parallel to $\ell$ through points of $C$ are also in $C$, which implies in particular that the boundary of $C$ contains also a line (and $F$ is therefore affine along the projection on $\R^n$ of this line).  Thus, $C$ must contain extreme points, and hence exposed points. Let $(x_0,F(x_0))$ be an exposed point of the boundary of $C$. This means there exists $z_0\in \R^n$ such that  the convex function $x\to G(x):=F(x)-x\cdot z_0 \in \R\cup\{\infty\}$ has a \emph{unique} minimum at $x_0$. In turn this implies 
 there exists $a,b>0$ such that 
 \begin{equation}\label{eq:lowerG}
      G(x)\geq a|x| -b, \quad \forall x\in \R^n.
 \end{equation}
Indeed, since $m:= G(x_0)<G(y)$ for all $y\in \R^n$, we have, using the lower-semi continuity of $G$,  that $M:=\inf_{|\theta| = 1} G(x_0+\theta) >m$. A standard argument then gives~\eqref{eq:lowerG} with $a=M-m>0$: for $|x-x_0|>1$, write $x_0+\frac{x-x_0}{|x-x_0|}  = \frac{1}{|x-x_0|} x + \big(1- \frac{1}{|x-x_0|}\big)x_0 $, and, for $|x-x_0|\le1$, simply invoke that $G$ is lower-bounded.   

Property~\eqref{eq:lowerG} implies $\int_{\R^n}e^{-(F(x) - x\cdot z_0)}dx < \infty $, and so $\domain(L(e^{-F}))\not\equiv \emptyset$.
\end{proof}

 We need the following elementary fact concerning convergence of sequences of convex functions.

\begin{fact}\label{f:uniform_eventually}
    Let $F_k : \R^n \to \R\cup\{+\infty\}$ be a sequence of convex functions converging pointwise to a convex function $F : \R^n \to \R\cup\{+\infty\}$ on a dense subset $S$ of $\R^n$. Assume $\inte(\domain(F))\neq \emptyset$. Then:
    \begin{enumerate}
        \item $F_k\to F$ uniformly on any compact subset of $\inte( \domain(F))$.
        \item $F_k$ converges to $+\infty$ uniformly on any compact subset of $\inte\{F=+\infty\}=\inte(\domain(F)^c)$.
    \end{enumerate}
\end{fact}

\begin{proof}
Let $x_0\in \inte( \domain(F))$ and let $z_1, \ldots, z_{n+1}$ be affinely independent points such that $x_0$ is in the interior of the simplex $\Delta=\conv(z_1, \ldots, z_{n+1})$ and $\Delta$ is contained in the interior of the domain of $F$. Furthermore, we may move if needed the $z_i$'s a little bit so that they belong to $S$. By definition, there exists $N$ such that for every $k\ge N$, $F_k(z_i)< \infty$ for $i=1, \ldots, n+1$. This implies, by convexity, that $F_k$ is finite on $\Delta$ for every $k\ge N$. On the interior of $\Delta$, we have a sequence $(F_k)_{k\ge N}$ of finite convex functions converging pointwise to a finite convex function $F$, and it is classical that the convergence is then uniform over compact subsets, see~\cite[Theorem 10.8]{RTR70}. We thus pick any ball $B$ of positive radius containing $x_0$ and contained in the interior of $\Delta$. 

Now, for the second case, let $x_0\in\inte\{F=+\infty\}$. Fix $z_0\in\inte(\domain(F))$ and
$r>0$ such that $z_0+rB_2^n\subset \inte(\domain(F))$ and consider the set $C=\conv(z_0+\frac{r}{2}B_2^n,x_0)$. Since $x_0$ is in the interior of the complement of the domain of $F$, there exists a point $y\in S\cap C\cap \inte\{F=+\infty\}$, that can therefore be written as \begin{equation}
\label{eq:formula_y}
y=(1-\lambda) z_1 + \lambda x_0
\end{equation} for some  $z_1\in z_0+\frac{r}{2}B_2^n$ and $\lambda\in (0,1)$. Let us prove that $F_k$ converges uniformly to $F$ on the ball (included in $\{F=+\infty\}$) given by $$B:=x_0 + \frac{r}{2}\left(\frac{1}{\lambda}-1\right)B_2^n.$$ 
For any $x\in B$ let us define $z=z(x)=\frac{y-\lambda x}{1-\lambda}$, so that $y=(1-\lambda )z + \lambda x$. Then $z\in z_1+\frac{r}{2}B_2^n\subset z_0+rB_2^n \subset \inte(\domain(F))$. Indeed, one has 
\[
z-z_1=\frac{y-\lambda x}{1-\lambda}-\frac{y-\lambda x_0}{1-\lambda}=\frac{\lambda}{1-\lambda}(x_0-x)\in \frac{r}{2}B_2^n.
\]
Then the convexity of $F$ and $F_k$ give that
$$\
\frac{1}{\lambda}F(y) \le \left(\frac{1}{\lambda}-1\right) F(z) + F(x)\quad 
\hbox{and}\quad 
\frac{1}{\lambda}F_k(y) \le \left(\frac{1}{\lambda}-1\right) F_k(z) + F_k(x).$$
The first inequality ensures that $F(x)=F(y)=\infty$, since $F(z)<\infty$, and so 
$B\subset\{F=+\infty\}$. The second inequality ensures that $F_k(x)\to +\infty$ uniformly on $B$, since $F_k(z)\to F(z) <+\infty$ uniformly in $z_0+rB_2^n$ (from case 1) and $F_k(y)\to F(y)=+\infty$. 
\end{proof}

The previous fact can be formulated in terms of log-concave functions. We complete it with a useful fact concerning limits of integrals of log-concave functions. Such fact is stated in~\cite[Lemma 3.2]{AAKM04}, although the needed assumption that the pointwise limit is $\not\equiv 0$ is not explicitly stated there. Also, the proof in~\cite{AAKM04} somehow assumes that the limit is strictly positive, so we prefer to include a proof here.   

\begin{fact}
\label{fact:limit_integrals}
    Let $(g_k)_k$ be a sequence of log-concave functions converging pointwise on a dense subset of $\R^n$ to a log-concave function $g$ with $g\not\equiv 0$. Then, $g_k$ converges to $g$ uniformly on any compact subset of $\inte(\supp (g))$ and on any compact compact subset of $\R^n\setminus \supp (g)$, hence almost everywhere. 
    
    Moreover, 
    $$\int g_k \to \int g.$$
    In particular, $Lg_k \to Lg$ pointwise on $\R^n$.
\end{fact}
\begin{proof}
The uniform convergence follows from Fact~\ref{f:uniform_eventually} applied to the convex functions $-\log(g_k)$. This immediately implies that, if $\int g =\infty,$ then the claimed convergence of integrals holds, since $\int g = \int_{\inte(\supp(g))} g$ and so for every $C>0$ we can find a compact set $K$ in the interior of the support of $g$ such that $\int_K g \ge C$. 
Thus, we assume that $\int g<\infty$. Since $g_k\to g$ pointwise outside the boundary of $\{g=0\}$, the convergence occurs almost everywhere.

Since $g \not\equiv 0$, there exists a closed ball in $\interior(\supp(g))$.  Without loss of generality, we may assume that this ball is $B_2^n$. Then, by continuity of $g$, there exists $m,M>0$ such that $m< g(x)< M$ for all $x$ in this ball, and, since $g_k$ converges uniformly on $B_2^n$, there exists $k_0$ such that for all $k\ge k_0$ and for all $x\in B_2^n$ one has $m<g_k(x)<M$.

Since $g$ is log-concave and integrable, we have that $\sup_{|x|>R}g(x) \to 0$ as $R \to \infty$, by Fact~\ref{f:integrable}. Thus, we can find $R$ such that $|x|\ge R$ implies $g(x)<\frac{m}{e}$. Define the annulus $K_{R}=\{x;R\leq |x|\leq R+1\}$. Let $z_1,\dots, z_N\notin\partial (\supp(g))$ be such that   $K_R$ is covered by the balls $z_i+\frac{1}{2}B_2^n$, for $1\le i\le N$. Thus, there exists $k_1$ such that for every $1\le i\le N$ and for every $k\ge k_1$ one has $g_k(z_i)\le \frac{m}{e}$. Introducing the sphere $S_{2R}:=\{x;\; |x|=2R\}$, we have that $S_{2R}\subset 2K_R\subset \cup(2 z_i + B_2^n)$. For every $x\in S_{2R}$ there exists $1\le i\le N$ such that $x-2z_i\in B_2^n$. Consequently, 
$g_k(z_i)^2\ge g_k(2z_i-x)g_k(x)$ and for every $x\in S_{2R}$ and for every $k\ge k_1$ we get
\[
g_k(x)\le\frac{g_k(z_i)^2}{g_k(2z_i-x)}\le \frac{m}{e^2}.
\]
Using again the log-concavity of $g_k$, we deduce that for every $x$ such that $|x|\ge 2R$ one has 
\[
g_k(x)\le g_k(0)\left(\frac{g_k(2Rx/|x|)}{g_k(0)}\right)^\frac{|x|}{2R}\le Me^{-\frac{|x|}{R}}.
\]
The same argument also shows that for every $x$ such that $|x|\ge1$ and for every $k\ge k_1$ one has $g_k(x)\le g_k(0)(g_k(x/|x|)/g_k(0))^{|x|}\le M \left(\frac{M}{m}\right)^{|x|}$, hence a uniform bound for $|x|\le 2R$. 
The convergence of integrals then follows from dominated convergence.

    The 'in particular' follows from applying the result to the log-concave function $y\mapsto g_k(y) e^{x\cdot y}$.

\end{proof}

\subsection{Double Laplace transform and proof of Proposition~\ref{prop:sant}}

The goal of this section is to analyse
$$z\mapsto L(\lap(f))(z)=\intn \lap(\tau_{z/q} f)(x)\, dx$$
and, in particular, its infimum and the characterization of the point $\sant(f)$ for a general nonnegative function $f$. Since $\lap(f)$ is always a log-concave function when $f\not\equiv 0$, much of the argument relies on classical properties of the Laplace transform of log-concave functions we have just obtained above. 

It is maybe instructive to analyse some simple examples. Let us recall that $q<0$.  
\begin{itemize}
    \item If $f=1_{[a,b]}$ on $\R$ where $a<b$. Then we have 
$
\lap f(x)= \left(\int_{a}^b e^{x\, y}\, dy\right)^q=\left(\frac{e^{bx}-e^{ax}}{x}\right)^q,
$
and so the support of $\lap(f)$ is equal to $\R$, while
\[
L(\lap f)(z)=\int_{\R} e^{xz}\left(\frac{e^{bx}-e^{ax}}{x}\right)^qdx \in [0, \infty]
\]
has as domain $\domain(L(\lap(f)) = \big((-q)a\, ,\, (-q) b\big)$, and $L\lap(f)$ tends to $\infty$ at the points $-qa$ and $-qb$. In particular $L(\lap(f))$  attains its minimum. 
    \item If $f=1_{[a, \infty)}$, then $\lap(f)(x) = 1_{]-\infty, 0]} (-x)^{-q}e^{q a x}$ and 
$$L(\lap(f))(z) = \frac{\Gamma(1-q)}{(z+qa)^{1-q}} \quad \textrm{ if } z> -qa \quad \textrm{ and } \quad \infty \quad \textrm{otherwise},$$
has domain $(-qa, \infty)$. Note that $0$ is not in the interior of (the convex hull of) the support of $\lap f$ and that we have that $\inf L(\lap f) = 0$, by letting $z\to \infty$. 
    \item If $f(x)=e^{-\alpha |x|^2}$ for some $\alpha>0$, then $\lap(f)(x) = c\, e^{-\beta |x|^2}$ and $L\lap(f)(z)= C\, e^{\eta |z|^2}$ for some constants $c,C, \beta, \eta >0$. In particular $L\lap(f)$ attains its minimum at the origin. 
    \item If $f(x)=\big(\frac{1}{1+x^2}\big)^\alpha$ on $\R$, $\alpha>0$,   then, $\lap(f)(0)=\Big(\int f^{1/p}\Big)^q  $ and $\lap(f)(x)=0$ otherwise. Thus, $\lap(f)\equiv 0$ and so $L\lap(f)(z)=0$ for every $z$.
\end{itemize}   
One can also play with variants of the previous examples using~\eqref{eq:trans} and~\eqref{eq:trans2}.

We next study, when $f\not\equiv 0$, the domain of the Laplace transform of $\lap(f),$
$$\domain(L\lap(f))=  \left\{z \in \R^n\; ; \; \intn \lap(\tau_{z/q} f) <\infty\right\}.$$
We denote by $\intco(A)$ the interior of the convex hull of a set $A\subset \R^n$.

\begin{prop}\label{prop:finiteLp}
Let $f:\R^n\to [0,\infty)$ be a nonnegative function with $f\not\equiv 0$. We have
$$\domain(L\lap(f))\supseteq (-q) \intco(\supp(f)).$$ 
If $f$ is log-concave and integrable, there is equality in the previous inclusion. 
\end{prop}

Note that, in dimension $n\ge 2$, the example of an indicator of a half-space, which is a log-concave function, shows that the inclusion above can be strict if we don't assume integrability. 

\begin{proof}
If $\lap(f)\equiv 0$, then $\domain (L\lap(f))=\R^n$ and the claim is trivial. We now suppose that $\lap (f) \not\equiv 0$. Note that 
$$-q z+ (-q)\, \intco(\supp(f)) =(-q) \Big(\intco(z +\supp(f))\Big) = (-q)\intco (\supp(\tau_{z} f )).$$ 
So, without loss of generality, it suffices to prove that 
\begin{equation}
\label{eq:intco_implication}
    0 \in \intco(\supp(f) )\Longrightarrow  \int_{\R^n} {\mathcal{L}}_p( f) <\infty.
\end{equation}
Assume that $\int_{\R^n} {\mathcal{L}}_p( f) =\infty$. Since $\lap(f)$ is log-concave, from Fact~\ref{f:integrable}, there exists $a\in\R^n$ and a direction $u\in \s^{n-1}$ such that $\lap(f)(a+tu)$ does not tend to $0$ as $t\to+\infty$. But the log-concavity of $\lap(f)$ implies that $\lap(f)(a+tu)$ has a limit in $(0,+\infty]$. We deduce that $L(f^{1/p}) (a+tu) \to \ell \in [0, \infty)$, as $t\to+\infty$.

Let us introduce the nonnegative function $g(x)=e^{a\cdot x} f^{1/p}(x)$, for which we have $L(g)(tu)\to \ell$. Introduce the open half-space 
$H^+=\{x\in \R^n\, ; \; x\cdot u >0\}$. We have
$$Lg (tu) =\intn g(x)\, e^{t u\cdot x}\, dx  \ge \int_{H^+} g(x)\,   e^{t u\cdot x}\, dx .
$$
By monotone convergence we have
$$ \lim_{t\to \infty}\int_{H^+} g(x) \, e^{t u\cdot x} \,dx = \int_{H^+} \big( g\times  \infty ). $$
Therefore $g=0$ almost everywhere on $H^+$, and so
$\textrm{supp} (g)\subseteq \R^n \setminus H^+$. 
 Since $\supp(g)=\supp(f)$, by definition, we also have
that $\supp(f)$ is contained in the convex set $\R^n \setminus H^+$, whose interior does not contain the origin. Thus, we have proved that $0\notin \intco(\supp(f))$. 

Assuming that $f$ is log-concave and integrable, we can prove the converse implication in~\eqref{eq:intco_implication}. Notice that, by hypothesis, we have, from \eqref{eq:majo_log_concave}, that $f(x) \leq Ce^{-c|x|}$. Thus, $f^{1/p}(x)\leq C^\prime e^{-c^\prime|x|}$, and so $f^{1/p}$ is integrable and more generally, $\lap(f)>0$ in a neighbourhood of zero. 

If $0 \notin \intco(\supp(f))$, then there exists $u\in \s^{n-1}$ such that $\supp(f) \subseteq H^-:=\{x\in\R^n; x\cdot u \leq 0\}$. Then, 
$$L(f^\frac 1p)(tu)=\int_{H^-}f^{1/p}(x)e^{tu\cdot x}dx.$$
We deduce from dominated convergence that
$\lim_{t\to\infty} L(f^\frac 1p)(tu) = 0$. Therefore the log-concave function $\lap(f)\not\equiv 0$ does not tend to zero at infinity, and so $\intn \lap(f) = \infty$. 
\end{proof}

We now prove Proposition~\ref{prop:sant}. Actually, let us rephrase and complete this proposition. 

\begin{prop} \label{prop:exist}

Let $f:\R^n\to\R^+$ such that $f\not\equiv 0$. Then:
\begin{enumerate} \item The convex function $L(\lap(f))$ has an open non-empty domain,
$$\domain(L\lap(f)) \neq \emptyset,$$ 
\item and the following assertions are equivalent:
\subitem{(i)} $\displaystyle \inf_z L(\lap(f))(z)\,  >0 $.
\subitem{(ii)}  $0$ lies in the interior of the support of $\lap(f)$.
\subitem{(iii)}  $L \lap(f)$ tends to $\infty$ on the boundary of its domain.
\subitem{(iv)} $\lap(f)\not \equiv 0$ and $\inf_z  L(\lap(f))(qz)$   is attained at a (unique) point $z_0$ in the domain of $L\lap f$.
\subitem{(v)} $\lap (f) \not \equiv 0$ and there exists $s_0\in\R^n$ such that $qs_0\in\domain(L\lap(f))$ and
$${\rm bar} (\lap(f)(x) e^{q x\cdot s_0})=0.$$

Moreover, the points $z_0$ in $(iv)$ and  $s_0$ in $(v)$ are equal (and denoted by $\sant(f)$). 
\end{enumerate}
\end{prop}

\begin{proof}
We give two different proofs of the property $1.$

It can be seen as a direct consequence of Proposition~\ref{prop:finiteLp}. Indeed, since $f\not\equiv 0$, the support $\supp(f)$ cannot be contained in an affine hyperplane, since it is of non-zero measure, and therefore its convex hull is of non-empty interior. 

We give an alternative proof that relies on Fact~\ref{fact:finite}. 
Define the function $F = -\log \lap(f) = (-q)\log L(f^{1/p}):\R^n\to\R\cup\{\infty\}.$ This function is convex, continuous on the interior of its domain and, from Fatou's lemma, it is lower-semi-continuous on its domain. Assume that the domain of  $L\lap(f)=L(e^{-F})$ is empty. 
Then, by Fact~\ref{fact:finite}, $F$ must be  affine along some line. Therefore, there exists $u\in\s^{n-1}$, $v\in u^\bot$ and $\alpha,\beta\in \R$ such that for all $t\in \R$,
$\log L(f^{1/p})(tu+v) = \alpha t+ \beta$, or equivalently
\begin{equation}
\label{eq:contradiction}
\int_{\R^n}f^{1/p}(y)e^{(tu+v)\cdot y}dy = e^{\alpha t +\beta}.
\end{equation}
We now use Fubini's theorem; by defining the function
$$g(s)=\int_{u^\perp}f^{1/p}(su+w)e^{v\cdot w}dw,\qquad\forall s\in \R,$$
we deduce, from \eqref{eq:contradiction}, that
$$\int_{\R}g(s+\alpha) e^{ts}ds=e^{\beta}.$$
This means that  the Laplace transform of $s\mapsto g(s+\alpha)$ is constant, which implies that  $g\equiv 0$, since otherwise the Laplace transform is strictly convex. In turn, this implies that $f$ is zero almost everywhere on $u^\perp+su$ for almost all $s$, contradicting $f\not\equiv 0$. 
\medskip

We now move on to the proof of 2. We assume that $\lap(f)\not\equiv 0$ since all properties require it, explicitly or implicitly. From Fact~\ref{fact:mf} applied to the log-concave function $\lap(f)$,
we have $(i)\Leftrightarrow(ii)\Rightarrow(iii)\Rightarrow (iv)\Leftrightarrow (v)$, using the smoothness and (strict) convexity of $L(\lap(f))$. 
The presence of $q$ is immaterial in the Laplace transform. Obviously, $(iii)$ implies $(i)$. Finally, assume $(v)$. This implies that $z\mapsto \log L(\lap(f))(z)$  has a critical point at $s_0$, and this point has therefore to be the (unique) point where the function attains its infimum. In particular we have $(i)$.  
\end{proof}

Let us give a slightly different way of passing from $(v)$ to $(ii)$ in the previous Proposition. Assume that there exists $s_0$ in the domain of $L(\lap(f))$ such that the barycenter of $x\mapsto g(x):= e^{s_0\cdot x} \lap(f)(x) $ is at the origin. Then $g$ is a log-concave function with $0<\int g < \infty$ and, from that, one can deduce that $g$ has a barycenter that must be in the interior of its support. We conclude by noting that the support of $g$ is equal to the support of $\lap(f)$. 

In order to treat extremal situations, which have to be consistent with the statement in Theorem~\ref{t:main_sant}, it is convenient to have one more sufficient (and sometimes necessary) condition for the infimum to vanish. 

\begin{fact}\label{final:f}
Let $f$ be a nonnegative function on $\R^n$ and  $p\in (0,1)$. Then we have
$$\int_{\R^n} f = \infty \Longrightarrow \inf_z \int_{\R^n}\lap(\tau_z f) = 0.$$
If $f$ is log-concave, the converse is also true. 
\end{fact}
\begin{proof} Note that we can enforce $f\not\equiv 0$, since 
$f\equiv 0 \Rightarrow \lap(f)\equiv \infty$. 

 Let us assume that $\inf_z \int_{\R^n}\lap(\tau_z f) > 0$, and prove that $f$ is integrable. The hypothesis gives $\inf_z L(\lap(f))(z)>0$, hence, from Proposition~\ref{prop:exist}, we must have $0\in {\rm int}(\supp(\lap(f)))$. Thus we can find $r>0$ so that $[-r,r]^n \subset {\rm int}(\supp(\lap(f)))$. We then obtain for $x\in [-r,r]$ that $L(f^\frac{1}{p})(x) \in [0,\infty)$, that is, $[-r,r]^n \subset {\rm int}(\domain(L(f^\frac{1}{p})))$. Since $\lap(f)$ is log-concave, it is continuous on the interior of its domain. Consequently, $L(f^\frac{1}{p})$ is bounded on $[-r,r]^n$. Observe then that
    \begin{align*}
        \infty &> \int_{[-r,r]^n}L(f^\frac{1}{p})(x)dx = \int_{\R^n}f^\frac{1}{p}(y) \int_{\{\|x\|_\infty \le r\}} e^{x\cdot y} dx dy
        \\
        &=2^n\int_{\R^n}f^\frac{1}{p}(y) \prod_{i=1}^n \frac{\sinh(ry_i)}{y_i} dy.
    \end{align*}
    We also note that
    $\int_{\R}\left(\frac{\sinh(y)}{y}\right)^\frac{p}{p-1}dy < \infty$, since $p/(p-1)<0$. Finally, we obtain, from H\"older's inequality,
    \begin{align*}
        \int_{\R^n}f(y)dy &= \int_{\R^n}\left[f^\frac{1}{p}(y) \prod_{i=1}^n \frac{\sinh(ry_i)}{y_i}\right]^p\left[\prod_{i=1}^n \frac{\sinh(ry_i)}{y_i}\right]^{-p}dy
        \\
        &\leq \left(\int_{\R^n}f^\frac{1}{p}(y) \prod_{i=1}^n \frac{\sinh(ry_i)}{y_i} dy\right)^p\left(\int_{\R^n}\left[\prod_{i=1}^n \frac{\sinh(ry_i)}{y_i}\right]^{\frac{p}{p-1}}dy\right)^{1-p},
    \end{align*}
    and the claim follows from the above observations.

    Conversely, assume that $f$ is log-concave, with $0<\intn f <\infty$. This implies that $f(x)\le a e^{-b |x|}$ for some constants $a,b>0$. Thus, we have a similar upper bound for $f^{1/p}$, and this implies that $L(f^{1/p})$ is finite in a neighborhood of the origin. In turn, this means that $\lap(f)$ is strictly positive in a neighborhood of the origin, and therefore, by Proposition~\ref{prop:exist}, the infimum of $L\lap(f)$ is strictly positive. 
\end{proof}

Finally, we verify that, if $f$ is even and $\lap(f)\not\equiv 0$, then its Laplace-Santal\'o point is the origin.
    \begin{prop}
    \label{p:even}
    Let $f$ be a nonnegative, \emph{even} function 
 on $\R^n$. Then, we have
  $$ \Big(\int_{\R^n}f \Big)\, \inf_z  \Big(\int_{\R^n}{\mathcal  L}_p (\tau_z f) \Big)^{-p/q}
  = \Big(\int_{\R^n}f \Big)\, \Big(\int_{\R^n}{\mathcal  L}_p (f) \Big)^{-p/q}.
$$ 
 Moreover, if $\lap(f) \not\equiv 0$,  then $\sant(f)=0$.
\end{prop}
\begin{proof}
If $\lap(f)\equiv 0$, then the equality is trivial. Let us assume that $\lap(f) \not\equiv 0$.  The support of the log-concave function $\lap(f)$ is a convex set with non-empty interior, and symmetric, so it contains the origin in its interior. Thus, by Proposition~\ref{prop:exist},  the infimum is attained at a unique point $\sant(f)$. On the other hand, $\lap(f)$ has a well defined barycenter, which is the origin. Therefore we find that $\sant(f)=0$, and the equality also follows in this case.  
\end{proof}
We note that we could also have alternatively used Proposition~\ref{prop:exist} $(v)$ to prove Proposition~\ref{p:even}.

\subsection{The case $p=0$: the (essential) polar}

In order to pass to the limit as $p\to 0^+$ and obtain results for the polar transform and the usual volume product, we need to establish the analogues of
 Propositions~\ref{prop:exist} and \ref{prop:finiteLp} when $p=0$. This is easier and essentially well-known. Actually, it can be deduced from the case $p>0$ by passing to the limit, although this is not the shortest path. We include the proofs for completeness.  
 
We set
\begin{equation}
f^\square(x):= {\rm ess}\!\inf_{y\in\R^n}\frac{e^{-x\cdot y}}{f(y)}.
\label{eq:ess_polar}
\end{equation}
For fixed $x,x'\in \R^n$ and $\lambda\in (0,1)$ we have $f^\square((1-\lambda)x+\lambda x')\ge f^\square(x)^{1-\lambda} f^\square(x')^\lambda$. So when $f\not\equiv 0$, the function $f^\square:\R^n\to [0,\infty)$ is a log-concave function, that coincides with the polar function $f^\circ$ when $f$ is continuous or when $f$ is log-concave. 
We next recall that \eqref{eq:limit_p} holds, i.e. that the limit of $\lap(f)\left(\frac{x}{p}\right)$ exists and equals $f^\square$ as defined in \eqref{eq:ess_polar}, except possibly on a set of measure zero, with no restrictions on $f$.
\begin{fact}
\label{f:supp_converge}
    Let $f:\R^n\to [0,\infty)$ be a function such that $f \not\equiv 0$. Then for every $x\in \R^n$ such that $x\in\interior(\supp (f^\square))$ or $f^\square(x)=0$, and thus for almost-every $x\in \R^n$, 
\begin{equation}
        f^\square(x)=\lim_{p\to0^+}\lap(f)\left(\frac{x}{p}\right).
    \end{equation}
\end{fact}
\begin{proof}
    Recall first the following classical fact. Let $g$ be any function on $\R^n$ such that $g\not\equiv 0$. Then, if $\|g\|_{\infty} = \infty$, one has $\|g\|_{p^\prime} \to \infty$ (with no assumptions on $g$) as $p^\prime \to \infty$. If $\|g\|_{\infty} < \infty$ and there exists $p_0$ such that $\|g\|_{p_0} <\infty$, then $\|g\|_{p^\prime} \to \|g\|_\infty$ as $p^\prime \to \infty$.

Fix $p\in (0,1)$. For a fixed $x\in \R^n$, we set $g_x(y) = e^{x\cdot y}f(y)$ and write $p^\prime = 1/p$. With this notation, we have $\lap(f)\left(\frac{x}{p}\right) = \|g_x\|_{p^\prime}^\frac{1}{p-1}$. Clearly, the exponent $\frac{1}{p-1}$ will not play such a serious role in the analysis. The key observation is that $\|g_x\|_{\infty} = \frac{1}{f^\square(x)}$. 

We will break the proof into two parts: $x_0\in \{f^\square =0\}$ and $x_0\in {\rm int}({\rm supp} (f^\square))$. The first case is easy: if $x_0\in \{f^\square =0\}$, then, $\|g_{x_0}\|_\infty = \infty$, and thus, as stated early, it holds $\|g_{x_0}\|_{p^\prime} \to \infty$. We deduce that $\lap(f)\left(\frac{x_0}{p}\right) \to 0 (=f^\square(x_0)).$
Next, assume that $x_0\in {\rm int}({\rm supp} (f^\square))$. It suffices to show
\begin{equation}\label{eq:intfsquare}
    x_0\in {\rm int}({\rm supp} (f^\square)) \Longrightarrow \int f(x) e^{x\cdot x_0}\, dx < \infty,
\end{equation}
which gives that $\|g_{x_0}\|_1<\infty$ and allows to conclude as recalled above. 
To prove~\eqref{eq:intfsquare}, note that $\tau_{-x_0} (f^\square) = (g_{x_0})^\square$, so it suffices to consider the case $x_0=0$. Since $f^\square$ is log-concave, we have then $f^\square(x) \ge c 1_{C B_2^n}(x)$ for some constant $c,C>0$. Let $S$ be a countable dense subset of the sphere. We have, by definition,  that, for almost all $y\in \R^n$, for all $x\in S$, 
$$f(y) \le \frac1c e^{- C\,  x\cdot y}.$$
By letting $x$ approach $y/|y|$ we find 
$$f(y) \le \frac1c e^{- C |y|} \quad \mbox{ \rm for a.e }\, y.$$
Therefore $f$  is integrable, as wanted. 
\end{proof}

Note that we still have the property that, for any $z,x\in \R^n$,
\begin{equation}\label{eq:trans_square}
    (\tau_z f)^\square (x)= e^{-x\cdot z} f^\square (x).
\end{equation}

\begin{prop}
\label{prop:ess}
Let $f$ be a nonnegative function on $\R^n$, with $f\not\equiv 0$. Then:
    \begin{enumerate}
        \item We have that
        $$\domain(L (f^\square))\supseteq \intco(\supp(f)).$$ 
If $f$ is log-concave and integrable, there is equality in the previous assumption. Anyway, $$\domain(L(f^\square)) \neq \emptyset.$$ 
\item Additionally, the following assertions are equivalent:
\subitem{(i)} $\displaystyle \inf_z L(f^\square)(z)\,  >0 $.
\subitem{(ii)}  $0$ lies in the interior of the support of $f^\square$.
\subitem{(iii)}  $L (f^\square)$ tends to $\infty$ on the boundary of its domain.
\subitem{(iv)} $f^\square\not \equiv 0$ and $\inf_z  L(f^\square)(-z)$   is attained at a (unique) point $z_0$ in the domain of $L(f^\square)$.
\subitem{(v)} $f^\square \not \equiv 0$ and there exists $-s_0\in\R^n$ such that $s_0$ is in the domain of $L(f^\square)$ and
$${\rm bar} (f^\square(x) e^{-x\cdot s_0})=0.$$

Moreover, the points $z_0$ in $(iv)$ and  $s_0$ in $(v)$ are equal (and denoted by $s_{\rm ess}(f)$).
    \end{enumerate}

\end{prop}

\begin{proof}
    Like before, we start with the proof of $1.$ If $f^\square\equiv 0$ then the claim is trivial,  so in the sequel we assume that $f^\square \not\equiv 0$, that is the support of the log-concave function $f^\square$ has non-empty interior. Like in the proof of Proposition~\ref{prop:finiteLp}, we may shift $f$ so that it suffices to prove
    $$ 0 \in \intco(\supp(f) )\Longrightarrow  \int_{\R^n} f^\square <\infty.$$
When $\int_{\R^n} f^\square =\infty$, it follows from Fact~\ref{f:integrable}, since $f^\square$ is log-concave, that there exists $u,v\in \R^n$, $v\neq 0$, and $\ell\in(0,+\infty]$ such that
$f^\square(u+ t v)\to \ell$ as $t\to \infty$. Consider
the open half-space $H^-=\{ y\in \R^n\; ; \; y\cdot v <0\}$. By definition, we have that, for every $k\in \N$,
$$\frac{e^{u\cdot y}}{f(y)}\, e^{ k v \cdot y}
\ge f^\square(u + k v), \quad \textrm{ for almost all $y\in \R^n$.}$$
But, when  $y\in H^-$, $e^{k v\cdot y} \to 0$ as $k\to \infty$. Therefore, $f(y) =0$ for almost all $y\in H^-$, that is $\supp(f)\subseteq \{ x\in \R^n \; ; \; x\cdot v \ge 0\}$ and so $0 \notin \intco(\supp(f) )$.

We now give two proofs for the equality when $f$ is $\log$-concave and integrable.

We first note that it follows from Proposition~\ref{prop:finiteLp} by taking limits. Indeed, we set $D_p(f)= \domain\left(L\left(\lap(f)\left(\frac{\cdot}{p}\right)\right)\right)$ and $D(f)=\domain(L (f^\square))$. Assume $0\in D(f)$. From Fact~\ref{fact:limit_integrals}, there exists $\delta,N>0$ such that $\delta B_2^n \subset D_p(f)$ for $0<p<1/N$. From Proposition~\ref{prop:finiteLp}, 
$ D_p(f)= \frac{1}{1-p}\inte(\supp(f))$. Thus, $(1-p)\delta B_2^n \subset\inte(\supp(f))$ for every $0<p<1/N.$ Consequently, $0\in\inte(\delta B_2^n) \subset  \inte(\supp(f))$.

For a direct proof, note that for the integrable log-concave function $f\not\equiv0$, we have that $f(x) \leq Ce^{-c|x|}$,  and so in particular $f^\square>0$ in a neighborhood of zero.
If $0 \notin \intco(\supp(f))$, then there exists $u\in \s^{n-1}$ such that $\supp(f) \subseteq H^-:=\{x\in\R^n; x\cdot u \leq 0\}$. Then, 
$$f^\square(-tu) =  {\rm ess}\!\inf_{x\in H^-} \frac{e^{-t x \cdot u}}{f(x)}\ge 
\frac1C .$$
Therefore the log-concave function $f^\square \not\equiv 0$ does not tend to zero at infinity, hence $\intn f^\square = \infty$. 

For the proof of $2.$, this is, as before, a straightforward application of Fact~\ref{fact:mf} to the log-concave function $g=f^\square$.
\end{proof}

We next have the analogue of Fact~\ref{final:f} for $p=0$.

\begin{fact}
\label{f:final_2}
Let $f$ be a nonnegative function on $\R^n$. Then, we have
    $$\intn f = \infty \Longrightarrow \inf_z \int_{\R^n}(\tau_z f)^\square = 0.$$
If $f$ is log-concave, the converse is also true. 
\end{fact} 
\begin{proof}
If $f\equiv 0$ both implications are trivial, so we assume $f\not\equiv 0$. 

We will prove the contrapositive. By Proposition~\ref{prop:ess}, the function $f^\square$ contains the origin in the interior of its support. Then~\eqref{eq:intfsquare} with $x_0=0$ gives the desired conclusion $\int f<\infty$. Let us mention that one can alternatively  take the limit as $p\to0^+$ directly in Fact~\ref{final:f}, using Fact~\ref{fact:limit_integrals}, Proposition~\ref{prop:exist}, and Fact~\ref{f:supp_converge}.

Conversely, assume that $f$ is log-concave, $f\not\equiv 0$ and that $\intn f <\infty$. This means from \eqref{eq:majo_log_concave} that $f(x)\le a e^{-b |x|}$ for some constants $a,b>0$. This implies that $f^\circ$ and thus $f^\square$ is strictly positive in a neighborhood of the origin, and therefore, by Proposition~\ref{prop:ess}, the infimum of $Lf^\square$ is strictly positive. 
\end{proof}


\subsection{Continuity of the infimum and of the Laplace-Santal\'o point}

As we will have to perform several kind of approximations, we are led to investigate (weak) continuity properties for 
$$f \mapsto\inf_z L(\lap(f))(z)$$
and for $f \mapsto \sant(f)$. 

With this framework, we have the following useful result, which concerns $g\mapsto \inf Lg$ and will later be applied to $g=\lap(f)$,  along sequences of log-concave functions.

\begin{prop}\label{prop:cont_santalo}
    Let $g_1, g_2, \ldots g_\infty:\R^n \to [0,\infty)$ be  log-concave functions, all $\not\equiv 0$, such that, for almost all  $x\in \R^n$,
    $$g_k(x) \longrightarrow g_\infty(x).$$
    This implies that $Lg_k(z) \longrightarrow Lg_\infty(z)$ in $\R\cup\{\infty\}$ at every $z\in \R^n$.
   Moreover, if we assume the following two properties of non-degeneracy: the convex functions $g_k$, $k\in \N\cup\{\infty\}$, are proper, and
   $$\inf_z L g_k(z) >0, \qquad \forall k\in \N\cup\{\infty\},$$
    then we have
    $$ \inf_z Lg_k(z) \longrightarrow 
    \inf_z Lg_\infty(z). $$
    Furthermore, for each $k$, there exists a unique point $z(g_k)$ where the infimum $\inf_z L g_k(z)$ is attained and 
    $$z(g_k) \longrightarrow z(g_\infty).$$    
\end{prop}

The condition that the functions are proper will always be verified in our applications, where $g_k $ is some $\lap(f)$ with $f\not\equiv 0$, in virtue of Proposition~\ref{prop:exist}. The condition on the infimum being non zero will require more care. 

The above proposition follows from classical results in the theory of convex optimization (see, for example, the book by Rockafellar and Wets \cite{RW97}). These results are often stated using the more broad framework of \textit{epi-convergence}, and the reader may find the proposition as written here to not be stated so directly. Therefore, we provide a self-contained proof, without epi-convergence, for the reader's convenience. 

\begin{proof}
For ease of presentation, we set $g=g_\infty$. 
We recall from fact Fact~\ref{fact:limit_integrals} that $Lg_k(x) \to Lg(x)$ 
for every $x\in\R^n$.

  For each $k$, $Lg_k:\R^n\to \R\cup\{\infty\}$ is a lower-semi-continuous, strictly convex function, and by hypothesis, $\domain(Lg_k)\neq \emptyset$. Furthermore, the assumption that the infimum is not $0$ yields that the infimum of $Lg_k$ is obtained at a unique point, that we denote by $z(g_k)$, on its domain, which is open,  via Fact~\ref{fact:mf}. We break the remaining proof into two steps. First, we show that $\lim_{k\to\infty} z(g_k) =z(g)$. Finally, we show that
\begin{equation}\lim_{k\to\infty}\inf_z Lg_k(z) = \inf_z Lg(z).
    \label{eq:lim_of_inf}
    \end{equation}   
By Fact~\ref{f:uniform_eventually}, there exists a closed ball $B(z(g),\alpha)\subset \domain(Lg)$, $\alpha>0$, on which the convergence of $Lg_k$ to $Lg$ is uniform.  Recall domains are open here.  We parameterize the boundary of $B(z(g),\alpha)$ as $z(g)+\alpha\theta$ for $\theta\in\s^{n-1}$. We have that $Lg(z(g)+\alpha\theta)>Lg(z(g))$. 
Since $Lg$ is continuous on $B(z(g),\alpha)$, it follows that 
\[
r:=\inf_{\|x-z(g)\|=\alpha}(Lg(x)-Lg(z(g)))=\inf_{\theta\in \s^{n-1}}(Lg(z(g)+\alpha\theta)-Lg(z(g)))>0.
\]
Since the convergence is uniform on $B(z(g),\alpha)$, there exists $J$ such that for every $k\ge J$ and every $x\in B(z(g),\alpha)$ one has $|Lg_k(x)-Lg(x)|\le r/4$. Thus for every $\theta\in \s^{n-1}$ and $k\ge J$, one has 
\[
r_k:=\inf_{\theta\in \s^{n-1}}Lg_k(z(g)+\alpha\theta)-Lg_k(z(g))\ge \inf_{\theta\in \s^{n-1}}Lg(z(g)+\alpha\theta)-Lg(z(g))-\frac{r}{2}= \frac{r}{2}.
\]
By convexity, the map $t\mapsto \frac{Lg_k(z(g)+t\alpha\theta)-Lg_k(z(g))}{t}$ is non decreasing. Thus, for every $t\ge1$ one has 
\[
Lg_k(z(g)+t\alpha\theta)-Lg_k(z(g))\ge t(Lg_k(z(g)+\alpha\theta)-Lg_k(z(g)))\ge tr_k\ge \frac{tr}{2}.
\]
This implies that for any $x\notin B(z(g),\alpha)$, $Lg_k(x)\ge Lg_k(z(g))+r/2$, thus $z(g_k)\in B(z(g),\alpha)$ for every $k\ge J$. We have thus proved that, for any $\alpha$ such that $B(z(g),\alpha)\subset \domain(Lg)$, there exists $J$ such that, for every $k\ge J$, one has $\|z(g_k)-z(g)\|\le \alpha$. This proves that the sequence $\{z(g_k)\}_k$ converges to $z(g)$. 
    
Finally, we prove \eqref{eq:lim_of_inf}. For $k\ge J$, from the previous step of the proof, we have $z(g_k)\in B(z(g),\alpha)$; using that $Lg_k$ is minimal at $z(g_k)$ and $Lg$ is minimal at $z(g)$, we get
\begin{align*}
     Lg_k(z(g_k)) - Lg(z(g_k))\leq Lg_k(z(g_k))-Lg(z(g)) \leq Lg_k(z(g)) - Lg(z(g)) .
\end{align*}
It follows that $|Lg_k(z(g_k))-Lg(z(g))|\le\sup_{x\in B(z(g), \alpha)}|Lg_k(x)-Lg(x)|$, which converges to $0$ when $k$ tends to $+\infty$ by uniform convergence.
\end{proof}

\section{Semi-groups}
\label{sec:semi}

This section is devoted to the proof of the following result, which contains our main Theorem~\ref{t:main_2} in the case where  the function is bounded and compactly supported.

\begin{thm}
\label{t:main}
Let $f$ be a non-negative function bounded and compactly supported with $f \not \equiv 0$. 
Let $f_t$ be its evolution along the Fokker-Planck or the heat semigroup. 
For $p\in (0,1)$, define the function $Q(t,z)$, for $t>0$ and $z\in \R^n$, by
$$Q(t,z) := \log \intn \lap(\tau_z (f_t)) = \log \intn \lap(f_t)(x)\, e^{qz\cdot x} \, dx= \log L\lap(f_t)(qz).$$
Then it holds that
\begin{equation} \label{eq:mainineq}
\partial_t Q + \frac12\cdot\frac{p}{-q} |\nabla_z Q|^2 \ge 0, \qquad {\rm on }\ (0,\infty)\times \R^n .    
\end{equation}
As a consequence, the function
\begin{equation}\label{eq:monotone}
\alpha(t):= \inf L\lap(f_t)= \inf_z  \intn \lap(\tau_z (f_t))(x)\, dx    = \int_{\R^n}\lap(\tau_{\sant(f_t)}f_t)(x)\, dx \end{equation}
is increasing in $t>0$ and 
\begin{equation}\label{eq:bound_reg}
M_p(f) \le M_p(f_t) \le M_p(\gamma_1)=  \lim_{t\to \infty} M_p (f_t).
\end{equation}
\end{thm}
The $\frac12$ in~\eqref{eq:mainineq} comes from our normalization of Heat semi-groups with $\frac12 \Delta$ as generator, as we will see below. Although we will not use it here, it is interesting to note, for future investigations, that property~\eqref{eq:mainineq} points to Hamilton-Jacobi equations. 

 We first recall the main tools that will be used: the Fokker-Planck and the heat semi-groups, together with the variance Brascamp-Lieb inequality for log-concave densities, which plays a crucial role. Then, we will move to the proofs of Theorem~\ref{t:main}; we split this into two parts. First, we prove the inequality~\eqref{eq:mainineq}, which is really the core of the argument and is inspired by~\cite{NT24}. Second, we derive~\eqref{eq:monotone} and \eqref{eq:bound_reg}.

\subsection{Fokker-Planck or heat semi-groups}
For a (nonnegative) integrable function $f$, we define its Fokker-Planck flow as $P_0 f = f$, and, for $t>0$ by
\begin{eqnarray}
P_t f(x)&=&e^{nt/2}\int_{\R^n}f(y)e^\frac{-|e^{t/2}x-y|^2}{2(e^{t}-1)}\frac{dy}{(2\pi(e^{t}-1))^\frac{n}{2}} \notag \\
&= &
\left(\int_{\R^n} f(y) e^{\frac{e^{t/2}}{e^{t}-1} x\cdot y - \frac1{2(e^{t}-1)}|y|^2} \, dy\right)\, \frac{e^{-\frac1{1-e^{-t}} |x|^2/2}}{(2\pi(1-e^{-t}))^\frac{n}{2}} . \label{eq:FP}   
\end{eqnarray}
Under suitable integrability assumptions, it verifies the equation $\partial_t P_tf = \mathcal{D}^\star P_tf$, where $$\mathcal{D}^\star f= \frac12\big(\Delta f + {\rm div}_x (xf)\big).$$ 
It is well-known that $\mathcal{D}^\star$ can be seen as  the adjoint of $\mathcal{D}$ from the Ornstein-Uhlenbeck operator~\eqref{eq:orn}. 
Under mild assumptions (satisfied below), we have $P_t f(x)\to f(x)$ as $t\to 0^+$ and $P_tf(x) \to (\int_{\R^n} f) \frac{e^{-|x|^2/2}}{(2\pi)^{n/2}} $ as $t\to \infty$. 

It is standard to renormalize the Fokker-Planck flow into the heat flow, and vice-versa. Let us consider $E_t f$ defined, for $t>0$, by
\begin{equation}
\label{eq:convert}
E_t f(x) = (1+t)^{-n/2} P_{\log(1 + t)}f\big( (1+t)^{-1/2} x\big).
\end{equation}
Then, it is readily checked that $E_t f$ has the following integral representation
\begin{equation}
\label{eq:heat_eq}
E_tf(y)=\frac{1}{(2\pi t)^{n/2}}(f\ast \gamma_{t})(y)=\intn f(u) \, e^{-|y-u|^2/(2t)}\, \frac{du}{(2\pi t)^{n/2}}, \qquad \forall y\in \R^n,
\end{equation}
and, under suitable integrability assumptions,  follows the heat equation
$$\partial_t E_t f (x)= \frac12 \Delta_x E_t f (x).$$
What is more surprising is that this renormalization scales perfectly with respect to the definition of $\lap$. Indeed, by change of variables, and using that $p$ and $q$ are conjugate, we see that, for any fixed $z\in \R^n$,
$$\intn \lap (\tau_z(E_tf))  = \intn \lap (\tau_{(1+t)^{-1/2} z} (P_{\log(1 + t)}f)).$$ 
In particular,
\begin{equation}\label{eq:eqFPh}
\inf_z \int \lap (\tau_z(E_t f))
= \inf_z \int \lap (\tau_z(P_{\log(1 + t)}f)).      
\end{equation}
And also, if we define
$$Q(t,z)= \log \intn \lap (\tau_z(E_t f)) \quad {\rm and}\quad \widetilde Q(t,z)= \log \intn \lap (\tau_z(P_tf)),$$
then we have $Q(t,z)= \widetilde Q(\log(1+t) , (1+t)^{-1/2} z)$  and therefore, for any constant $c$
$$(\partial_t Q + c|\nabla Q |^2) (t,z) = \frac1{1+t} ( \partial_t \widetilde Q + c|\nabla \widetilde Q |^2) (\log(1+t), (1+t)^{-1/2} z).$$
Consequently, proving~\eqref{eq:mainineq} along the Fokker-Planck or heat semi-group is totally equivalent. 

\subsection{Preliminary Inequalities}
We will need some well-known inequalities to prove our main results. Assume that we are given a $C^2$ smooth positive integrable function $h$. We denote by $\mu_h$ the corresponding probability measure,
$$d\mu_h(x)=\frac{h(x)}{\int_{\R^n} h}\, dx.$$

The first, essential, tool is the variance Brascamp-Lieb inequality, proven in \cite{BL76}.  We recall that the variance of a function $g$  with respect to a probability measure $\mu$ is given by
$$\mathrm{Var}_\mu g:=\int_{\R^n}|g|^2d\mu(x) - \left(\int_{\R^n}gd\mu(x)\right)^2.$$

\begin{lem}
Let $h$ be a strictly log-concave function. Then, for any locally Lipschitz $g\in L^2(\R^n, \mu)$, one has
\begin{equation}
\mathrm{Var}_{\mu_h} g \leq \int_{\R^n}\left( \nabla g\cdot (\nabla^2 (-\log h))^{-1}\nabla g \right) d\mu_h(x).
\label{eq:bl}
\end{equation}
\end{lem}

The second tool is much simpler and relates information and covariance (it is often refer as the Cram\'er-Rao inequality, see, for example, \cite[Equation 20]{ELS20}). Recall that the covariance matrix associated to a probability measure $\mu$ is given by
\begin{equation}
\label{eq:covar}
    \mathrm{cov}(\mu):=\int_{\R^n} z \otimes z d\mu (z)-\left(\int_{\mathbb{R}^n} z d\mu(z)\right) \otimes\left(\int_{\R^n} z d\mu(z)\right).
\end{equation}

\begin{lem}
\label{l:cr} Assume that $\mu_h$ has finite variance, $\intn |z|^2 \, d\mu_h(z) <\infty$ and finite Fisher information, $\intn \frac{|\nabla h|^2}{h} < \infty$.  Then, the following matrix inequality holds:
$$I(\mu_h) \geq \mathrm{cov}(\mu_h)^{-1},$$
where $I(\mu_h):= \frac1{\intn h} \, 
\intn \frac{\nabla h \otimes \nabla h}{h}$ is the (Fisher) information matrix.
\end{lem}

\subsection{Proof of~\eqref{eq:mainineq} in Theorem~\ref{t:main}}
Let $\mathcal C$ be the set of smooth functions $h$ on $\R^n$ such that, for some constants $c_0,C_0, c_1,C_1>0$, it holds
\begin{equation}\label{eq:classC}
C_0 \, e^{-c_0 |x|^2} \le g(x) \le C_1\,  e^{-c_1|x|^2} \quad \forall x \in \R^n ,    
\end{equation}
where $g$ is $h$ or any (partial) derivative of $h$ of order $\le 2$. 

Let us fix, as in Theorem~\ref{t:main}, a nonnegative function $f$ that is bounded and compactly supported, with $\intn f>0$. Let us denote here $$f_t:=E_t f,$$ the evolution of $f$ along the heat semi-group.
It is readily checked that, for any $t>0$, the functions $f_t$, $f_t^{1/p}$ and $\lap(f_t)$ belong to the class $\mathcal C$ defined above. Note that both $L(f^{1/p})$ and $L(f_t^{1/p})$ are everywhere finite, which means that the corresponding $\lap$ functionals are non-zero everywhere, and thus, by Proposition~\ref{prop:exist} we have
\begin{equation}
    \inf_z L(\lap(f))(z) >0 \qquad\textrm{and}\qquad  \inf_z L(\lap(f_t))(z) >0,
\label{eq:finiteft}
\end{equation}
and these infimum are attained at a unique point. 
We also have, for fixed $t>0$,  that for any $z$, the function $x\mapsto \lap(f_t)(x) e^{z\cdot x}$ also belongs to $\mathcal C$, and, in particular, 
\begin{equation}
    \domain L(\lap(f_t)) = \R^n.
\label{eq:domaint}
\end{equation}
Since $f_t$ is continuous and strictly positive, we could also invoke Proposition~\ref{prop:finiteLp}.
The function $z \mapsto \log L(\lap(f_t))(z)$ is a smooth convex function on  $\R^n$, having a positive Hessian everywhere.

Recall the definition, on $(0,\infty)\times \R^n$, of
$$Q(t,z)= \log \intn \lap(\tau_z (f_t))(x)\, dx =\log \intn  \lap(f_t)(x) \, e^{q z\cdot x}\, dx .$$
We aim at computing
$$\partial_t  Q (t,z) =\frac{1}{\intn\lap(\tau_z (f_t)) }   \intn \partial_t \big[\lap(f_t)(x) \big]\, e^{q z\cdot x}\, dx.$$
The derivative in time can indeed be moved inside the integral because the constants in~\eqref{eq:classC} for $g=\lap(\tau_z(f_t))$ depend continuously on $t>0$ and hence are locally bounded.
Introduce 
$$h_t := f_t^{1/p} \quad{\rm and} \quad H_t(x) := L(h_t)(x)=\intn h_t(y) \, e^{x\cdot y}\, dy.$$
Since $\lap(\tau_z f_t)(x)=H_t^q(x) e^{q\, z\cdot x}$, we have
\begin{equation}
\begin{split}
\label{eq:H_and_Q}
    \partial_t Q(t,z) &= \frac{1}{\intn H_t^q(x)\,  e^{q\, z\cdot x}\, dx}   \intn  \partial_t [H_t^q(x)] e^{q z\cdot x}\, dx\\
&= \frac{q}{\intn H_t^q(x)\,  e^{q\, z\cdot x}dx}   \intn  H_t(x)^{q-1}\partial_t [H_t(x)] e^{q z\cdot x}\, dx. 
\end{split}
\end{equation}
Thus, we must compute $\partial_t H_t$.
Let us observe that $h_t$ and $H_t$ are smooth, and moreover that $h_t$ and $H_t^q$ belong to the class $\mathcal C$. This suffices to justify the computations below. Recall also, for later reference,  that $\log H_t$ is a smooth strictly convex function on $\R^n$. 

First note that
$$\partial_t h_t = \frac{1}{2}\left(\Delta h_t + (p-1) \frac{|\nabla h_t|^2}{h_t}\right),$$
so 
$$2\, \partial_t H_t (x) = \intn \Delta h_t(y)\,  e^{x\cdot y}\, dy + (p-1) \intn \frac{|\nabla h_t(y)|^2}{h_t(y)}  e^{x\cdot y}\, dy .  $$
For the first term we have, by integration by parts the following relation (a classical property of the Laplace transform), 
$$\intn \Delta h_t(y)\,  e^{x\cdot y}\, dy = |x|^2 H_t(x),$$
and, for the second term,
\begin{eqnarray*}
\intn \frac{|\nabla h_t(y)|^2}{h_t(y)}  e^{x\cdot y}\, dy &=& \intn \frac{|\nabla_y [h_t(y)\,  e^{x\cdot y}]|^2}{h_t(x)\, e^{x\cdot y}}  \, dy  -2 \intn \nabla h_t(y)\cdot x \, e^{x\cdot y}\, dy  - \intn h_t(y) |x|^2 e^{x\cdot y}\, dy \\
& =  & \intn \frac{|\nabla_y [h_t(y)\,  e^{x\cdot y}]|^2}{h_t(x)e^{x\cdot y}}  \, dy +2|x|^2 H_t(x) - |x|^2 H_t(x).
\end{eqnarray*}
Therefore, we have the following nice formula:
\begin{equation}
\label{eq:nice_formula}
2\, \partial_t H_t (x) = p |x|^2 H_t(x) + (p-1)I (h_t(y) e^{x\cdot y}) \, H_t(x), \end{equation}
where 
$$I (h_t(y) e^{x\cdot y}) := \intn \frac{|\nabla_y [h_t(y) e^{x\cdot y}]|^2}{h_t(y)e^{x\cdot y}}  \, \frac{dy}{\intn h_t(u)e^{x\cdot u}\, du} $$
is the Fisher information of the probability density $h_t(y) e^{x\cdot y} \, \frac{dy}{\intn h_t(u)e^{x\cdot u}\, du}$.
To sum up, by introducing for fixed $t>0$ and $z$, the log-concave probability measure $\mu_{t,z}$ on $\R^n$ given by
$$d\mu_{t,z}(x) := H_t^q(x) e^{q\, z\cdot x} \frac{dx}{\intn H_t^q(x) e^{q\, z\cdot x} \, dx} = e^{- (-q)\big(\log H_t(x)+ z\cdot x\big)} \frac{dx}{\intn H_t^q(x) e^{q\, z\cdot x} \, dx },   $$
and inserting \eqref{eq:nice_formula} into \eqref{eq:H_and_Q}, we obtain
\begin{equation}\label{eq:maineq}
\frac2q \, \partial_t Q(t,z) = p \intn |x|^2 \, d\mu_{t,z}(x)  +  (p-1)\intn  I(h_t(y) e^{x\cdot y}) \, d\mu_{t,z}(x) .     
\end{equation}
The idea of regrouping the terms in this way and the forthcoming arguments are taken from the work of Nakamura and Tsuji; it is beautiful and new. 

For the first term in~\eqref{eq:maineq}, we invoke the variance Brascamp-Lieb inequality \eqref{eq:bl} for the (strictly and smooth) log-concave measure $\mu_{t,z}$ and the linear functions $x\mapsto x_i$, $i=1,\ldots, n$: 
$$\intn x_i^2 \, d\mu_{t,z} - \Big(\intn x_i \, d\mu_{t,z} \Big)^2 \le \intn \big( \nabla_x ^2 (-q\log H_t) \big)^{-1} e_i \cdot e_i \, d\mu_{t,z},$$
where $(e_i)$ refers to the canonical basis of $\R^n$. Therefore we find,
$$ \intn |x|^2 \, d\mu_{t,z}(x) \le  \Big|\intn x \, d\mu_{t,z} (x)\Big|^2  +  \left(-\frac 1 q\right)\intn {\rm Tr}[(\nabla_x^2\log H_t)^{-1}] d\mu_{t,z}(x).$$
For the second term, we bound the term inside the integral for $x$ fixed, using the Cram\'er-Rao inequality, Lemma~\ref{l:cr}; the inverse of the covariance matrix is dominated by the information matrix, which implies in particular that
$$I(h_t(y) e^{x\cdot y}) \ge {\rm Tr}\left[{\rm Cov}\left( \frac{h_t(y) e^{x\cdot y}}{\intn h_t(s) e^{x\cdot s}\, ds}\right) ^{-1} \right] =  {\rm Tr}[(\nabla_x^2\log H_t)^{-1}] ,$$
since $\log(H_t)$ is the log-Laplace transform of $h_t$. 
Combining the previous inequalities with \eqref{eq:maineq}, and noting that $p-1<0$ and $p-1=\frac pq$, we find
$$\frac2q \, \partial_t Q(t,z)  \le p \Big|\intn x \, d\mu_{t,z} (x)\Big|^2 = \frac{p}{q^2}  | \nabla_z  Q (t,z)|^2,$$
where the equality follows from the definition of $Q(t,\cdot)$ as the log-Laplace transform at $q z$ of $\lap(f_t)$. 
This ends the proof of~\eqref{eq:mainineq}.

\subsection{Proof of~\eqref{eq:monotone} in  Theorem~\ref{t:main}}
\label{ss:proof_main}

We continue with the same notations and assumptions as in the previous section. As we said, since $\lap(f_t)$ belongs to the class $\mathcal C$, the function $F_t$ defined by
\begin{equation}F_t(z):= Q(t,z)=\log\intn \lap(f_t)(x) \, e^{q\, z\cdot x}\, dx
\label{eq:log_time_laplace}
\end{equation}
is a smooth convex function, with domain $\R^n$, with positive Hessian everywhere, tending to infinity at infinity, and attaining its minimum at a unique point denoted by $\sant(f_t)$ characterized by $\nabla F_t(\sant(f_t))=0$. Combining this with the regularity of $Q$, we immediately see that $\sant(f_t)$ is at least $C^1$-smooth in $t>0$. 
Consider the function
\begin{equation}
\label{eq:alpha_time}
\alpha(t)=F_t(\sant(f_t)) = Q(t,\sant(f_t)).\end{equation}
Since $\nabla F_t (\sant(f_t))=\nabla_z Q(t, \sant(f_t))=0$, we have
\begin{eqnarray*}
\alpha'(t)&=& \partial_t Q (t,\sant(f_t)) + \nabla F_t (\sant(f_t)) \cdot \partial_t \sant(f_t) \\
& = & \partial_t Q(t,\sant(f_t))\\
& \ge & 0,
\end{eqnarray*}
where we used the fact that  $\sant(f_t)$ has the special property that $0 = p \Big|\intn x \, d\mu_{t,z} (x)\Big|^2 = \frac{p}{q^2}  | \nabla_z  Q (t,z)|^2$ at $z=\sant(f_t)$. We therefore have the desired monotonicity for the infimum. 

\subsection{Proof of~\eqref{eq:bound_reg} in Theorem~\ref{t:main}}
We only need to examine the limits when $t\to 0$ and $t\to \infty$.
Keeping the notations of the previous section, that is,
$$f_t:=E_t f \quad \text{and} \quad \alpha(t)=\inf_z \intn \lap(\tau_z(f_t)),$$
we have, for every $t_0>0$,
$$\lim_{t\to0^+} \alpha(t)\le\alpha(t_0)
\le\lim_{t\to \infty}\alpha(t).$$
We must now relate these limits with the corresponding functionals associated with our endpoint functions $f$ and $\gamma_1$. 

Let us consider first the limit as $t\to 0^+$. Since $f\in L^1(\R^n)$, we have that $f_t$ converges to $f$ almost everywhere~\cite[Chapter 3, Theorem 2.1]{Stein}.  Moreover, since $f$ is bounded and compactly supported, we have that for every $0<t<1/10$, say, and every $y\in \R^n$,  $f_t(y) \le c \, e^{-|y|^2}$. So by dominated convergence, we have that $L(f_{t}^{1/p})(x)\to L(f^{1/p}) (x)$ and 
$$\lap(f_{t})(x) \to \lap(f)(x),$$ 
for every $x\in \R^n$, as $t\to 0$. 
Recalling~\eqref{eq:finiteft}, we can apply
 Proposition~\ref{prop:cont_santalo}
and obtain that $$\lim_{t\to 0} \alpha(t)= \inf_z\intn\lap(\tau_z f)=\alpha(0).$$
Note that we also have $\sant(f_{t}) \to \sant(f)
$.

We now move to the limit as $t\to \infty$. Here, it is more natural to work with the Fokker-Planck evolution of 
$f$, given by~\eqref{eq:FP}. We already saw in~\eqref{eq:eqFPh} that this is totally equivalent, up to a change in time in $\alpha$ that do not alter the limit. 
We ask the reader to forgive us for keeping the same notation $f_t$, that is
$$f_t=P_t f.$$
Using that $f$ is bounded and compactly supported, we see from the definition of $f_t$ that $f_t(x)\to c(f) \gamma_1(x) $ when $t\to \infty$, where $c(f)= (\intn f)/(2\pi)^{n/2}$. And we also have, for $t\ge 100$, say, that $f_t(x) \le C\,  e^{-|x|^2/4}$ for all $x\in \R^n$ and some constant $C>0$. Therefore, by dominated convergence, we get that $L(f_t^{1/p})(x)\to c(f)^{1/p} L(\gamma_1^{1/p})(x)$ and
$$\lap(f_t)(x) \to  \lap(c(f) \, \gamma_1) (x),$$
as $t\to \infty$, for all $x\in \R^n$.  As above, Proposition~\ref{prop:cont_santalo} gives that
$$\alpha(t) \to  \inf_z L(\lap(c(f) \, \gamma_1)) (z),$$
 with convergence of the corresponding Laplace-Santal\'o points. 
It is readily checked using~\eqref{eq:limit_fixed} or invoking Proposition~\ref{p:even}, that $\sant(c(f)\gamma_1)=0$. So, we have proven that,  when $t\to \infty$,
$$\alpha(t) \to  \intn \lap(c(f) \gamma_1) = c(f)^{q/p} \intn \lap (\gamma_1) , $$ 
as wanted. The result for $M_p(f_t)$ follows by multiplying $\alpha(t)^{-p/q}$ by 
$$\intn f_t = \intn f = c(f) \intn \gamma_1 \in (0, \infty).$$

\section{Proof of Theorem~\ref{t:main_2}, Theorem~\ref{t:main_sant} and Theorem~\ref{t:main_sant_2}}
\label{sec:proofs}

Here, we prove, by approximation,  our general statements regarding the maximum of the $p$-volume product from Theorem~\ref{t:main}. We then explain how results for the classical volume product are obtained, by letting $p\to 0^+$. 

\subsection{Proof of Theorem~\ref{t:main_2} and Theorem~\ref{t:main_sant}}

We begin with the useful observation that if $\inf L(\lap(f_t))>0$ at some time $t\ge 0$, then  $\inf L(\lap(f_T))>0$ at all times $T\ge t$.

\begin{lem}
\label{c:domains}
    Let $f$ be a non-negative function, $f\not\equiv 0$,  and let $f_t=E_t f$ be its heat flow evolution. 
    Then for $0\le s \le t$,
    \begin{equation*}
    \supp (\lap(f_s)) \subseteq \supp (\lap(f_t)) ,
    \end{equation*}
    and consequently
    $$\inf L\lap(f_t)=0 \Longrightarrow
    \forall s\in [0,t], \ \inf L\lap(f_s)=0.$$
\end{lem}

\begin{proof}Since $f_t = (f_s)_{t-s}$, it suffices to consider the case $s=0<t$. It is readily checked from the definition of the heat semi-group, since $(2\pi t)^{-{n/2}} \intn e^{x\cdot u}\, e^{-|u-y|^2/2t}\, du  = e^{t\frac{|x|^2}{2}}\, e^{x\cdot y} $,  that for any heat flow evolution $f_t = E_t f$,
$$L (f_t) = e^{t\frac{|x|^2}{2}} L f.$$
From Jensen's inequality, we have that $(f_t)^{1/p}\le (f^{1/p})_t$ , from which we derive that
$$L((f_t)^\frac{1}{p})(x) \leq e^{t\frac{|x|^2}{2}}L(f^\frac1p)(x).$$
Therefore, we have
$$\lap(f_t)(x) \ge e^{qt\frac{|x|^2}{2}}\lap (f)(x).$$
Since  $e^{q\frac{t|x|^2}{2}}>0$, the claim on the supports follows. To conclude, we use the equivalence of $(i)$ and $(ii)$ in Proposition~\ref{prop:exist}.
\end{proof}

We can now move to the proof of Theorem~\ref{t:main_2}.

\begin{proof}[Proof of Theorem~\ref{t:main_2}]
 If $f\equiv 0$ then $M_p(f_t)=0$ at all times and the result is trivial. When $\intn f = \infty $, we have $\intn f_t = \infty$ at all times and so, by  Fact~\ref{final:f}, we also have that $M_p(f_t)=0$ at all times.
 Thus, in the sequel, we can assume that 
 $$0<\intn f < \infty.$$
 Therefore, the result for $M_p$ amounts to proving that for any $t>0$,
\begin{equation}
\label{eq:main_ineq}
\inf L \lap(f) \le 
\inf L \lap(f_t), 
\end{equation}
and 
\begin{equation}
\label{eq:main_ineq_infty}
      \inf L \lap(f_t) \le  \Big( (2\pi)^{-n/2} \intn  f \Big)^{q/p} \, \intn \lap (\gamma_1). 
\end{equation}
Indeed, by the semi-group property $f_t = (f_s)_{t-s}$, inequality~\eqref{eq:main_ineq} improves to 
$$\inf L\lap(f_s)\le 
\inf L \lap(f_t), $$
for any $0\le s < t$.

So let us fix $t>0$. 
We consider the following, classical,  bounded compactly supported approximation procedure: for $k\in \N$, set
$$f^{(k)}:= f  \cdot 1_{|f|\le k} \cdot 1_{|x|\le k}.$$
We implicitly take $k$ large enough so that $f^{(k)}\not\equiv 0$.
By Theorem~\ref{t:main}, we have that
\begin{equation}
\inf L \lap(f^{(k)})
\le \inf L \lap((f^{(k)})_t).
\label{eq:approx_ineq}
\end{equation}

\begin{fact}\label{fact:approx}
    Let $f$ be a nonnegative function, $f\not\equiv 0$, and let $s\in [0,\infty)$. If $\inf L \lap(f_s)>0$ then, as $k\to \infty$, we have
    $$\inf L \lap((f^{(k)})_s) \to \inf L \lap(f_s).$$
\end{fact}
\begin{proof}
 We have $f^{(k)}(x)\nearrow f(x)$ at every $x\in \R^n$ as $k\to \infty$, and so, by monotone convergence, we have, 
$(f^{k})_s \nearrow f_s,$ 
and 
$$\lap((f^{(k)})_s) \searrow \lap(f_s).$$
Note that this implies that 
$\inf L\lap((f^{(k)})_s)  \ge \inf L\lap(f_s) >0$.
So, we can apply Proposition~\ref{prop:cont_santalo} to conclude. 
\end{proof}
We continue with the proof of~\eqref{eq:main_ineq}. We can assume that
$$\inf L\lap(f)>0 \quad
\textrm{and}\quad 
\inf L\lap(f_t)>0.$$
Indeed, if $\inf L\lap(f)=0$, there is nothing to prove. If $\inf L\lap(f_t)=0$, it follows from Lemma~\ref{c:domains}
that  $\inf L\lap(f)=0$, and so the inequality is true in this case too.  So, we can apply the Fact~\ref{fact:approx} with $s=0$ and $s=t$,  and conclude by passing  to the limit in~\eqref{eq:approx_ineq}.

There remains to prove~\eqref{eq:main_ineq_infty}. We can again assume that $\inf L\lap(f_t)>0$. We have, from Theorem~\ref{t:main}, that
$$\inf L \lap((f^{(k)})_t)\le \lim_{s\to\infty} L \lap((f^{(k)})_s) = \Big( (2\pi)^{-n/2} \intn f^{(k)}\Big)^{q/p} \intn \lap (\gamma_1).
$$
We obtain~\eqref{eq:main_ineq_infty}, by letting $k\to \infty$, using Fact~\ref{fact:approx} in the left-hand side, and monotone convergence in the right-hand side.  

Let us note, for consistency, that  $\intn f= \intn f_t < \infty$, and so $\Big(\intn  f \Big)^{q/p}>0$, when  $\inf L \lap(f_t)>0$, because of Fact~\ref{final:f}. 
\end{proof}

\subsection{Proof of Theorem~\ref{t:main_sant_2}}

Let $f$ be a nonnegative function with $f\not\equiv 0$.

Suppose $\lap(f)$ has barycenter at zero. This implies, by definition,  that $\lap(f) \not\equiv 0$.  By Proposition~\ref{prop:exist} $(v)$, we know that 
$$\inf_z \intn \lap(\tau_zf) = \intn \lap(f),$$
and the result follows by Theorem~\ref{t:main_sant}. Note for consistency that $\lap(f)(0)>0$, rewrites as $\intn f <\infty$. 

Assume next that $0$ is the barycenter of $f$. If  $\intn \lap(f) =  0$, there is nothing to prove, so we can assume that $\lap(f)\not\equiv 0$. Since the barycenter belongs to the interior of the convex hull of the (essential) support, we have, by Proposition~\ref{prop:finiteLp}, that $L\lap(f)(0)<\infty$. Thus, the log-concave function $\lap(f)$ is integrable. 
Note that, for any fixed vector $s\in \R^n$, we have 
$$\lap(f(x) e^{s\cdot x}) = \tau_{-s/p} \lap(f) ,$$
so we can find $s_0$ such that the barycenter of $\lap(f(x) e^{s_0\cdot x})$ is at the origin: just take $s_0 = p\,  {\rm bar}(\lap(f))$.  Applying Proposition~\ref{prop:exist} and Theorem~\ref{t:main_sant} to $x\mapsto f(x) e^{s_0\cdot x}$, we deduce that
$$\int f(x) e^{s_0\cdot x} \, dx \intn \lap(f) \le M_p(\gamma_1).$$
We conclude with the following usual (in the context of the Blaschke-Santal\'o inequality) observation:
$$\intn f(x)\,  e^{s_0\cdot x} \, dx \ge \intn f(x)\, dx + \intn (s_0\cdot x)\,  f(x)\,dx = \intn f(x)\, dx.$$ 

\subsection{Limits as $p\to 0^+$ and proof of Corollary~\ref{cor:volume_heat}}

In this section, we verify that Theorem~\ref{t:main_sant} implies the same result for the usual volume product by letting $p\to0^+$. 

Via the semi-group property, it again suffices to show that
$$M(f)\leq M(f_t) \leq M(\gamma_1).$$
Without loss of generality, we will use the heat-semi group. 

The case when $\intn f = \infty$ follows from Fact~\ref{f:final_2}, and the case $f\equiv 0$ is trivial. From now we assume that $0<\intn f <\infty$, and so we are reduced to proving that, for $t>0$ fixed, 
\begin{equation}
    \inf L(f^\square) \le \inf L((f_t)^\square) \le (2\pi)^{n} \Big(\intn f\Big)^{-1}. 
\label{eq:square_goal}
\end{equation}
Observe first that, for $t\geq 0,$
$(f_t)^\square(x) \geq e^{-t\frac{|x|^2}{2}}f^\square(x)$ pointwise. Indeed,
\begin{align*}
    f_t(y)&=\intn f(u) \, e^{-|y-u|^2/(2t)}\, \frac{du}{(2\pi t)^{n/2}}
    \\
    &\leq \frac{1}{f^\square(x)} \intn e^{-x\cdot u} \, e^{-|y-u|^2/(2t)}\, \frac{du}{(2\pi t)^{n/2}} = \frac{1}{f^\square(x)} e^{t\frac{|x|^2}{2}}e^{-x\cdot y}.
\end{align*}
Re-arranging and taking the essential infimum over all $y$ yields the result.
More generally, from the semi-group property, this yields $t\geq s \geq 0$,
\begin{equation}
\label{eq:relate_new}
(f_t)^\square(x) \geq e^{-(t-s)\frac{|x|^2}{2}}(f_s)^\square(x).\end{equation} Thus, $0 \in \inte(\supp((f_s)^\square))$ implies $0 \in \inte(\supp((f_t)^\square))$. From Proposition~\ref{prop:ess}, this implies that, if $\inf L ((f_t)^\square) =0$, then $\inf L ((f_s)^\square) =0$ for every $0\leq s\leq t$. 

 Consequently, in order to prove~\eqref{eq:square_goal}, we may suppose that 
 \begin{equation}
 \inf L(f^\square) >0 \quad \textrm{and}\quad
 \inf L((f_t)^\square) >0  .
 \label{eq:square_inf_assumption}
 \end{equation}
 We claim that, for any $s\geq 0$ fixed, if $\inf L((f_s)^\square) >0,$ then $\inf L(\lap(f_s)(\frac{\cdot}{p})) >0$ for $p$ small enough. Indeed, the former is equivalent to $0\in \epsilon B_2^n\subset\inte (\supp ((f_s)^\square))$ for some $\epsilon>0$ via Proposition~\ref{prop:ess}. In particular, $(f_s)^\square$ is bounded away from zero on $\epsilon B_2^n$. Since by Fact~\ref{f:supp_converge},  $\lap(f_s)(\frac{\cdot}{p})$ is a sequence (indexed by $p=1/k$, $k\to \infty$, say) of log-concave functions converging almost-everywhere to a log-concave function, the (local) uniform convergence recalled in Fact~\ref{fact:limit_integrals} yields $0\in \epsilon B_2^n\subset\inte \left(\supp\left(\lap(f_s)\left(\frac{\cdot}{p}\right)\right)\right)$ for $p$ small enough, and the claim then follows from Proposition~\ref{prop:exist}. 
We have from Theorem~\ref{t:main_2} that, for $t>0$,
\begin{align*}\inf L\left(\lap(f)\left(\frac{\cdot}{p}\right)\right) &\leq \inf L\left(\lap(f_t)\left(\frac{\cdot}{p}\right)\right)  
\\
&\le \left([p(1-p)^\frac{1-p}{p}]^{-p\frac{n}{2}} \, (2\pi)^{n(1-p)}\, \Big(\intn f \Big)^{-1} \right)^{-q/p} .\end{align*}
In view of~\eqref{eq:square_inf_assumption} and of the claim of the paragraph above with $s=0$ and $s=t$, we can pass to the limit in $p\to 0^+$, thanks to repeated applications of Proposition~\ref{prop:cont_santalo} and conclude to~\eqref{eq:square_goal}.


\section{An application and open questions}
\label{sec:applications}

We begin with the derivation of Theorem~\ref{t:main_3} on hypercontractivity, and then briefly discuss other choices of centering along the heat flow.

\subsection{Hypercontractive inequalities}
We denote by $\gamma$ the standard Gaussian probability measure on $\R^n$. Recall that for $f$ nonnegative or in $L^1(\gamma)$ we can define the Ornstein-Uhlenbeck flow of $f$ by
$$U_t (f) (x) =\int_{\R^n}f(e^{-t}x+\sqrt{1-e^{-2t}}\, z)d\gamma(z).$$
Let us mention that, under appropriate assumptions, $f_t= U_t f$ satisfies, 
\begin{equation} \label{eq:orn} \frac{\partial}{\partial t} f_t:= \mathcal{D} f_t, \quad \text{where} \quad 
\mathcal{D} f:=\Delta f -  x\cdot \nabla f.\end{equation}
Nelson's hypercontractivity~\cite{EN73}, which quantifies the regularizing effect of the Ornstein-Uhlenbeck semi-group, states that
\begin{equation}\|U_s(f)\|_{L^{p_2}(\gamma)} \leq \|f\|_{L^{p_1}(\gamma)},
\label{eq:contract}
\end{equation}
when $1<p_1,p_2<\infty$ and $s>0$ satisfy
$\frac{p_2-1}{p_1-1}\leq e^{2s}$.
Borell's reverse hypercontractivity is then the fact \cite{Bor82} that \eqref{eq:contract} reverses when $p_1,p_2\in(-\infty,1)\setminus\{0\}$ satisfy the same relation.

With our indices~\eqref{eq:condpq}, it is readily checked, as observed by Nakamura and Tsuji~\cite{NT24},  that for any nonnegative $f$, we have the pointwise equality,
\begin{equation}
\label{eq:Usequiv1}
(U_sf(x))^q\times (2\pi p)^\frac{nq}{2}\gamma_1(x)=\lap(f^p \gamma_1)\left(\frac{x}{\sqrt{p|q|}}\right),
\end{equation}
with  $s=-\frac{1}{2}\log(1-p)$. 
As a consequence, we also have, for some constant $\tilde{c}_{n,p}>0$,
\begin{equation}
\label{eq:Usequiv2}
    \frac{\|f\|_{L^p(\gamma)}}{\|U_s f \|_{L^q(\gamma)}} =  \tilde{c}_{n,p}
\left( \Big( \intn f^p \gamma_1 \Big)\, \Big(\intn \lap(f^p \gamma_1)\Big)^{-p/q} \right)^{1/p}.
\end{equation}

Using~\eqref{eq:Usequiv1} and~\eqref{eq:Usequiv2}, we may reformulate Theorem~\ref{t:main_2} as a partial extension of Borell's reverse hypercontractivity beyond usual time:
     if we define $s$ via $p=1-e^{-2s}$ (and so $q=1-e^{2s}$), then, for any a nonnegative function $f$, if either $\int x\, f^p(x)\, d\gamma=0$ or $\int x\,  U_s(f)^q(x) \, d\gamma =0$, we have
$$\|U_s f\|_{L^q (\gamma)} \ge   \|f\|_{L^p(\gamma)}. $$
There is equality when $f^p\gamma_1$ is Gaussian.
The statement in  Theorem~\ref{t:main_3}, for $p_2\ge q$ and $p_1\le p$, then follows from Jensen's inequality.

\subsection{Open questions on Santal\'o curves}

Let us fix a nonnegative function $f$. Letting $f_t$ be, say, the Fokker-Planck evolution of $f$, a natural problem is to describe the following set of curves
$$\mathcal S_p(f) =\left\{s:\R^+\to \R^n\; ;\; \lim_{t\to\infty} s(t) = 0\; \textrm{ and } \; t\mapsto \mathcal \int \lap(\tau_{s(t)}f_t)\;  \textrm{ increases}\right\}.$$

It would be interesting to know if, after translating $f$, the single point zero curve belongs to $\mathcal S_p(f)$ or if the curves stay in a compact region that can be described. Actually, we don't even know if for the volume product, i.e. when $p=0$, $t\mapsto \int f_t \int (f_t)^\circ$ increases when $f$ or $f^\circ$ has barycenter $0$, say.  Note also that monotonicity ensures that $\int f_t \int (\tau_{s(t)}f)^\circ\le (2\pi)^n,$ and so $s(t)$ has to belong to the so-called 'Santal\'o region'~\cite{MW98, IW24} of $f_t$.

For instance, in the trivial case where $f(x)=e^{-|x-a|^2/2}$ is a Gaussian centered at $a\in \R^n$, then it is readily checked that $s(t)= e^{-t/2}a$ is the only element of $\mathcal S_p(f)$ for all $p\in[0,1)$. 

One can also look at the subset of curves in $\mathcal S_p(f)$ that make our argument work with fixed $p\in (0,1)$. In the proof of~\eqref{eq:monotone} in  Theorem~\ref{t:main}, given in Section~\ref{ss:proof_main}, we see that establishing the monotonicity in $t$ of $M_p(f_t)$ (assuming $f$ belongs to the class $\mathcal C$ defined in \eqref{eq:classC}) amounted to the fact that
\begin{eqnarray*}
\alpha'(t)&=& \partial_t Q (t,\sant(f_t)) + \nabla_z Q(t, \sant(f_t)) \cdot \partial_t \sant(f_t) \\
&\geq & \frac{p}{2q} |\nabla_z Q(t,\sant(f_t))|^2 + \nabla_z Q(t,\sant(f_t)) \cdot \partial_t \sant(f_t)\\
& = & 0,
\end{eqnarray*}
where $Q(t,z)= \log L(\lap(f_t))(qz)= \log \int \lap(\tau_z (f_t))$ is convex in $z$,  and  $\alpha(t) = Q(t, \sant(f_t)) $.
In particular, our choice of $z=\sant(f_t)$ ensured that
$$|\nabla_z Q(t,\sant(f_t))|^2 = 0 = \nabla_z Q (t,\sant(f_t)) \cdot \partial_t \sant(f_t).$$
But we see that we only need the curve to satisfy
\begin{equation}\label{eq:santcurve}
\frac{p}{2q} |\nabla_z Q(t,s(t))|^2 + \nabla_z Q (t,s(t)) \cdot s'(t)\geq 0,
\end{equation}
which leads to the subset $\widetilde{\mathcal S}_p(t) \subset \mathcal S_p(f)$
$$\widetilde{\mathcal S}_p(f) = \{s\in \mathcal C^1(\R^+, \R^n)\; ; \; \lim_{t\to\infty} s(t) = 0 \textrm{ and } \eqref{eq:santcurve} \textrm{ holds  at every } t>0, z\in \R^n\}.$$

A natural candidate, besides our $p$-Santal\'o points $s_p(f_t)$ themselves, for which both terms are zero, would be a curve satisfying the equation 
\begin{equation}
\label{eq:shift_alt}
 s'(t)=-\frac{p}{2q}\nabla_z Q(t, s(t))= \frac{p}{2}\int_{\R^n}x\frac{\lap(\tau_{s(t)}f_t)(x)}{\int_{\R^n}\lap(\tau_{s(t)}f_t)(x) dx}dx. 
\end{equation}
We leave open the following natural question: does this time-dependent gradient flow equation have at least one (smooth enough) solution on $\R^+$ with 
$\lim_{t\to \infty} s(t) = 0$? In the case $p=0$, the equation becomes
$$s'(t)=\frac12 \nabla\log L((f_t)^\square)(s(t)) .$$ 
In all cases, one probably needs to understand the time dependent log-Laplace transform $Q(t,z)$. 


\medskip

\noindent {\bf Acknowledgments: } We thank Hiroshi Tsuji for his useful comments on  the first draft of the present manuscript, and the anonymous referee for her or his useful remarks.

\bibliographystyle{acm}
\bibliography{references}

\begin{thebibliography}{10}

\bibitem{AAKM04}
{\sc Artstein, S., Klartag, B., and Milman, V.}
\newblock The {S}antal\'{o} point of a function, and a functional form of the {S}antal\'{o} inequality.
\newblock {\em Mathematika 51}, 1-2 (2004), 33--48 (2005).

\bibitem{Ball_thesis}
{\sc Ball, K.}
\newblock {\em Isometric problems in $l_p$ and sections of convex sets}.
\newblock PhD thesis, Cambridge, 1987.

\bibitem{Ball88}
{\sc Ball, K.}
\newblock Logarithmically concave functions and sections of convex sets in {${\bf R}^n$}.
\newblock {\em Studia Math. 88}, 1 (1988), 69--84.

\bibitem{BMR24}
{\sc Berndtsson, B., Mastrantonis, V., and Rubinstein, Y.~A.}
\newblock ${L}^p$-polarity, {M}ahler volumes, and the isotropic constant.
\newblock {\em Analysis \& PDE 17\/} (2024), 2179--2245.

\bibitem{Bor82}
{\sc Borell, C.}
\newblock Positivity improving operators and hypercontractivity.
\newblock {\em Math. Z. 180}, 2 (1982), 225--234.

\bibitem{BL76}
{\sc Brascamp, H.~J., and Lieb, E.~H.}
\newblock On extensions of the {B}runn-{M}inkowski and {P}r\'{e}kopa-{L}eindler theorems, including inequalities for log concave functions, and with an application to the diffusion equation.
\newblock {\em J. Functional Analysis 22}, 4 (1976), 366--389.

\bibitem{CEFPP15}
{\sc Cordero-Erausquin, D., Fradelizi, M., Paouris, G., and Pivovarov, P.}
\newblock Volume of the polar of random sets and shadow systems.
\newblock {\em Math. Ann. 362}, 3-4 (2015), 1305--1325.

\bibitem{CEGNT24}
{\sc Cordero-Erausquin, D., Gozlan, N., Nakamura, S., and Tsuji, H.}
\newblock Duality and heat flow.
\newblock {\em Preprint, {\tt arxiv:2403.15357}\/} (2024).

\bibitem{ELS20}
{\sc Eldan, R., Lehec, J., and Shenfeld, Y.}
\newblock Stability of the logarithmic {S}obolev inequality via the {F}\"{o}llmer process.
\newblock {\em Ann. Inst. Henri Poincar\'{e} Probab. Stat. 56}, 3 (2020), 2253--2269.

\bibitem{FGSZ24}
{\sc Fradelizi, M., Gozlan, N., Sadovsky, S., and Zugmeyer, S.}
\newblock Transport-entropy forms of {B}laschke-{S}antal{\'o} inequalities.
\newblock {\em Revista Matem\'atica Iberoamericana 14\/} (2024), 1917--1952.

\bibitem{FM07}
{\sc Fradelizi, M., and Meyer, M.}
\newblock Some functional forms of {B}laschke-{S}antal\'{o} inequality.
\newblock {\em Math. Z. 256}, 2 (2007), 379--395.

\bibitem{FMZ23}
{\sc Fradelizi, M., Meyer, M., and Zvavitch, A.}
\newblock Volume product.
\newblock In {\em Harmonic analysis and convexity}, vol.~9 of {\em Adv. Anal. Geom.} De Gruyter, Berlin, 2023, pp.~163--222.

\bibitem{GZ98}
{\sc Gardner, R.~J., and Zhang, G.}
\newblock Affine inequalities and radial mean bodies.
\newblock {\em Amer. J. Math. 120}, 3 (1998), 505--528.

\bibitem{HS09}
{\sc Haberl, C., and Schuster, F.~E.}
\newblock General {$L_p$} affine isoperimetric inequalities.
\newblock {\em J. Differential Geom. 83}, 1 (2009), 1--26.

\bibitem{HLPRY23_2}
{\sc Haddad, J., Langharst, D., Putterman, E., Roysdon, M., and Ye, D.}
\newblock Higher order ${L}^p$ isoperimetric and {S}obolev inequalities.
\newblock {\em J. Funct. Anal. To appear, Preprint available online at {\tt arXiv:2305.17468}\/} (2025+).

\bibitem{IW24}
{\sc Ivanov, G., and Werner, E.~M.}
\newblock Geometric representation of classes of concave functions and duality.
\newblock {\em J. Geom. Anal. 34}, 8 (2024), Paper No. 260, 25.

\bibitem{KS23}
{\sc Kalantzopoulos, P., and Saroglou, C.}
\newblock On a $j$-{S}antal\'{o} conjecture.
\newblock {\em Geom. Dedicata 217}, 5 (2023), Paper No. 91, 18.

\bibitem{BK06}
{\sc Klartag, B.}
\newblock On convex perturbations with a bounded isotropic constant.
\newblock {\em Geom. Funct. Anal. 16}, 6 (2006), 1274--1290.

\bibitem{BK07}
{\sc Klartag, B.}
\newblock Uniform almost sub-{G}aussian estimates for linear functionals on convex sets.
\newblock {\em Algebra i Analiz 19}, 1 (2007), 109--148.

\bibitem{KW22}
{\sc Kolesnikov, A.~V., and Werner, E.~M.}
\newblock Blaschke-{S}antal\'{o} inequality for many functions and geodesic barycenters of measures.
\newblock {\em Adv. Math. 396\/} (2022), Paper No. 108110, 44.

\bibitem{JL09}
{\sc Lehec, J.}
\newblock A direct proof of the functional {S}antal\'{o} inequality.
\newblock {\em C. R. Math. Acad. Sci. Paris 347}, 1-2 (2009), 55--58.

\bibitem{JL08}
{\sc Lehec, J.}
\newblock Partitions and functional {S}antal\'{o} inequalities.
\newblock {\em Arch. Math. (Basel) 92}, 1 (2009), 89--94.

\bibitem{LZ97}
{\sc Lutwak, E., and Zhang, G.}
\newblock Blaschke-{S}antal\'{o} inequalities.
\newblock {\em J. Differential Geom. 47}, 1 (1997), 1--16.

\bibitem{VM24_2}
{\sc Mastrantonis, V.}
\newblock $l^p$-legendre transforms and mahler integrals: Asymptotics and the fokker–planck heat flow.
\newblock {\em Preprint {\tt arxiv:2411.10439}\/} (2024).

\bibitem{VM24}
{\sc Mastrantonis, V.}
\newblock A {S}antal\'o inequality for the ${L}^p$-polar body.
\newblock {\em Preprint {\tt arxiv:2401.10836}\/} (2024).

\bibitem{MP90}
{\sc Meyer, M., and Pajor, A.}
\newblock On the {B}laschke-{S}antal\'{o} inequality.
\newblock {\em Arch. Math. (Basel) 55}, 1 (1990), 82--93.

\bibitem{MW98}
{\sc Meyer, M., and Werner, E.}
\newblock The {S}antal\'{o}-regions of a convex body.
\newblock {\em Trans. Amer. Math. Soc. 350}, 11 (1998), 4569--4591.

\bibitem{NT22}
{\sc Nakamura, S., and Tsuji, H.}
\newblock Hypercontractivity beyond {N}elson's time and its applications to {B}laschke-{S}antal\'o inequality and inverse {S}antal\'o inequality.
\newblock {\em Preprint {\tt arXiv:2212.02866}\/} (2022).

\bibitem{NT24}
{\sc Nakamura, S., and Tsuji, H.}
\newblock The functional volume product under heat flow.
\newblock {\em Preprint {\tt arxiv:2401.00427}\/} (2024).

\bibitem{EN73}
{\sc Nelson, E.}
\newblock The free {M}arkoff field.
\newblock {\em J. Functional Analysis 12\/} (1973), 211--227.

\bibitem{CMP85}
{\sc Petty, C.~M.}
\newblock Affine isoperimetric problems.
\newblock In {\em Discrete geometry and convexity ({N}ew {Y}ork, 1982)}, vol.~440. New York Acad. Sci., New York, 1985, pp.~113--127.

\bibitem{RTR70}
{\sc Rockafellar, R.~T.}
\newblock {\em Convex analysis}, vol.~No. 28 of {\em Princeton Mathematical Series}.
\newblock Princeton University Press, Princeton, NJ, 1970.

\bibitem{RW97}
{\sc Rockafellar, R.~T., and Wets, R. J.~B.}
\newblock {\em Variational Analysis}, vol.~317 of {\em Grundlehren der mathematischen Wissenshaften}.
\newblock Springer, Dordrecht, 1997.

\bibitem{SLA49}
{\sc Santal\'{o}, L.~A.}
\newblock An affine invariant for convex bodies of $n$-dimensional space.
\newblock {\em Portugal. Math. 8\/} (1949), 155--161.

\bibitem{Stein}
{\sc Stein, E.~M., and Shakarchi, R.}
\newblock {\em Real {A}nalysis: measure theory, integration, and {H}ilbert spaces}.
\newblock Princeton University Press, 2009.

\bibitem{T07}
{\sc Tao, T.}
\newblock Open question: the mahler conjecture on convex bodies.
\newblock {\em Blog's post https://terrytao.wordpress.com/2007/03/08/open-problem-the-mahler-conjecture-on-convex-bodies/\/} (2007).

\end{thebibliography}


\medskip
\noindent Dario Cordero-Erausquin 
\\
Institut de Math\'ematiques de Jussieu (IMJ-PRG), Sorbonne Universit\'e, CNRS,
4 Place Jussieu, 75252 Paris, France.
\\
E-mail address: dario.cordero@imj-prg.fr
\vspace{2mm}
\\
\noindent Matthieu Fradelizi
\\
LAMA, Univ Gustave Eiffel, Univ Paris Est Creteil, CNRS, F-77447 Marne-la-Vall\'ee, France.
\\
E-mail address: matthieu.fradelizi@univ-eiffel.fr
\vspace{2mm}
\\
\noindent Dylan Langharst
\\
Institut de Math\'ematiques de Jussieu (IMJ-PRG), Sorbonne Universit\'e, CNRS,
4 Place Jussieu, 75252 Paris, France.
\\
E-mail address: dylan.langharst@imj-prg.fr

\end{document}